\documentclass[12pt,leqno]{amsart}
\usepackage{a4,amssymb,amsthm,amscd,amsmath,verbatim,url,enumerate,mathdots,cleveref}
\usepackage[pdftex,usenames,dvipsnames]{color}
\usepackage{mathabx}
\usepackage{fancyhdr}
\usepackage{tikz}
\usepackage{tikz-cd}

\title{Square-reflexive polynomials}

\author{Karim Johannes Becher}
\author{Parul Gupta}

\address{University of Antwerp, Department of Mathematics, Middelheim\-laan~1, 2020 Antwerpen, Belgium.}
\email{karimjohannes.becher@uantwerpen.be}
\email{parul.gupta@uantwerpen.be}

\address{Technische Universit\"at Dresden, Institut f\"ur Algebra, 01062 Dresden, Germany.}
\address{IISER Pune, Dr.~Homi Bhabha Road, Pashan, Pune 411 008, India}
\email{parul.gupta@iiserpune.ac.in}
\thanks{This work was supported by the \emph{Fonds Wetenschappelijk Onderzoek -- Vlaanderen (FWO)} in the \emph{FWO Odysseus Programme} (project `{Explicit Methods in Quadratic Form Theory}'), the \emph{Bijzonder Onderzoeksfonds (BOF), University of Antwerp} (project BOF-DOCPRO-4, 2865), and the \emph{Science and Engeneenring Research Board  (SERB)}, India
(Grant CRG/2019/000271).}
\date{15.07.2021}

\newcommand{\la}{\langle}
\newcommand{\ra}{\rangle}

\newcommand{\qq}{\mathbb Q}

\newcommand{\cc}{\mathbb C}
\newcommand{\hh}{\mathbb H}
\newcommand{\nat}{\mathbb{N}} 
\newcommand{\zz}{\mathbb Z}

\newcommand{\N}{\mathsf N}
\newcommand{\mc}[1]{\mathcal{#1}}
\newcommand{\mf}[1]{\mathfrak{#1}}

\newcommand{\mg}[1]{{#1}^{\times}}
\newcommand{\sq}[1]{{#1}^{\times 2}}
\newcommand{\scg}[1]{\mg{#1}/\sq{#1}}

\newcommand{\lla}{\la\!\la}
\newcommand{\rra}{\ra\!\ra}

\newcommand{\ovl}{\overline}

\renewcommand{\min}{\mathsf{min}}
\renewcommand{\max}{\mathsf{max}}

\newcommand{\car}{\mathsf{char}}
\newcommand{\lra}{\rightarrow}

\newcommand{\supp}{\mathsf{Supp}}

\renewcommand{\k}{\mathsf{k}}

\newcommand{\mfm}{\mathfrak{m}}

\newcommand{\lc}{\mathsf{lc}}
\renewcommand{\deg}{\mathsf{deg}}
\renewcommand{\dim}{\mathsf{dim}}
\renewcommand{\sup}{\mathsf{sup}}
\renewcommand{\setminus}{\smallsetminus}
\renewcommand{\bmod}{\,\mathsf{mod}\,}
\newcommand{\vf}{\varphi}
\newcommand{\s}{\sigma}

\newcommand{\alg}{\mathsf{alg}}

\swapnumbers
\numberwithin{equation}{section}
\newtheorem*{thm*}{Theorem}
\newtheorem{thm}[equation]{Theorem}
\newtheorem*{thmA}{Theorem~A}
\newtheorem*{thmB}{Theorem~B}
\newtheorem*{thmC}{Theorem~C}
\newtheorem{prop}[equation]{Proposition}
\newtheorem{cor}[equation]{Corollary}

\newtheorem{lem}[equation]{Lemma}
\newtheorem{qu}[equation]{Question}
\newtheorem*{qu*}{Questions}
\theoremstyle{definition}

\newtheorem{ex}[equation]{Example}
\newtheorem{exs}[equation]{Examples}
\newtheorem{rem}[equation]{Remark}

\renewenvironment{proof}{\par\noindent {\em Proof:}}{\hfill$\Box$\medskip}
\theoremstyle{plain}

\begin{document}
\maketitle

\begin{abstract}
For a field $E$ of characteristic different from $2$ and cohomological $2$-dimension one, quadratic forms over the rational function field $E(X)$ are studied.
A characterisation in terms of polynomials in $E[X]$ is obtained for having that quadratic forms over $E(X)$ satisfy a local-global principle with respect to discrete valuations that are trivial on $E$.
In this way new elementary proofs for the local-global principle are achieved in the cases where $E$ is finite
 or pseudo-algebraically closed.
The study is complemented by various examples.

\medskip
\noindent
{\sc{Classification (MSC 2010):} 11E04, 12E05, 12E10, 12E20, 12E30} 

\end{abstract}

\section{Introduction}

Consider a field $E$ of characteristic different from $2$.
Quadratic forms over $E$ give us only a glimpse into the arithmetic of $E$.
The quadratic form theory of the rational function field $E(X)$ provides us with a much broader view on the arithmetic properties of $E$.
With this motivation we propose an analysis of quadratic form theory of $E(X)$ in terms of properties of polynomials in $E[X]$.

A quadratic form is \emph{isotropic} if it has a nontrivial zero, otherwise it is \emph{aniso\-tropic}.
In this article we study isotropy of quadratic forms over $E(X)$ in terms of local conditions.
We focus on the case where $E$ has cohomological $2$-dimension $0$~or~$1$, or equivalently (see~\cite[Theorem~6.1.8]{GS}),~such that, over every finite field extension of $E$, all $3$-dimensional quadratic forms are isotropic.
In particular, $-1$ is a sum of two squares in $E$ and hence $E$ has no field orderings.
Well-known examples of such fields $E$ are finite fields, function fields of curves over $\cc$, the field of Laurent series $\cc(\!(t)\!)$ and, more generally, $\mc{C}_1$-fields, in the terminology of Tsen-Lang theory (see \cite{Lan52} or \cite[Section 6.2]{GS}). 

By a \emph{$\zz$-valuation} on a field we mean a valuation with value group equal to~$\zz$.
We say that \emph{quadratic forms over $E(X)$ satisfy the local-global principle for isotropy}, if~an arbitary quadratic form over $E(X)$ is isotropic if and only if it has a nontrivial zero over all completions of $E(X)$ with respect to $\zz$-valuations on $E(X)$.
When $E$ is a finite field, then quadratic forms over $E(X)$ satisfy the local-global principle for isotropy.
This is a special case of the Hasse-Minkowski Theorem (see \cite[Chap.~VI, Theorem 66.1]{OM}). 
A typical proof of this fact uses that a finite field $E$ has a unique field extension of any given degree, together with  
 Kornblum's classical analogue for $E[X]$ of Dirichlet's theorem on primes in an arithmetic progression (see \Cref{T:Kornblum}).
However, in combination these two properties of $E$ are so restrictive that a proof using these ingredients does not suggest how to prove a local-global principle over $E(X)$ in other cases.

We restrict our focus to $\zz$-valuations on $E(X)$ which are trivial on $E$ (this condition is trivial when $E$ is finite), and we aim to characterise the fields $E$ such that quadratic forms over $E(X)$ satisfy the local-global principle for isotropy with respect to those valuations.

A detailed analysis of the finite field case led us to investigate a property for polynomials which is intimately related to the local-global principle over $E(X)$ when $E$ has cohomological $2$-dimension at most $1$. 
We call this property \emph{square-reflexivity}. It is introduced in \Cref{RS}, where different characterisations are given (\Cref{degreebound}), involving the Milnor $K$-group $\k_2E(X)$.

For a commutative ring $R$, we denote its multiplicative group by $\mg{R}$ and set $\sq{R}=\{x^2\mid x\in \mg{R}\}$, and we call an element of $R$ \emph{square-free} if it is not equal to $x^2y$ for any $x\in R\setminus \mg{R}$ and $y\in R$.

Consider now a polynomial $f\in E[X]$ and let $c \in \mg E$ be its leading coefficient.
We call $f$ \emph{square-reflexive} if $f$ is square-free and every element in the $E$-algebra $E[X]/(f)$ whose norm lies in $\sq{E}\cup c\sq{E}$  is given by a square times a polynomial $g\in E[X]$ such that $f$ is a square modulo $g$.
Every square-free polynomial in $E[X]$ of degree at most $2$ is square-reflexive (\Cref{SRuptodeg2}), but for polynomials of higher degree, square-reflexivity becomes a very restrictive condition.
For example, the presence of a discrete valuation or a field ordering on $E$ gives a natural obstruction for certain polynomials 
to be square-reflexive (\Cref{P:sqref-deg3} and \Cref{P:real-srp}). 
Even just having that in $E[X]$ every square-free polynomial of degree $4$ divides a square-reflexive polynomial implies that every  $3$-dimensional quadratic form over $E$ is isotropic (\Cref{L:sqr-deg4-k2triv}).

We call the field $E$ \emph{square-reflexive} if every square-free polynomial in $E[X]$ is square-reflexive.
In \Cref{LGPI4} we relate this to isotropy criteria  over $E(X)$:

\begin{thmA}
The field $E$ is square-reflexive if and only if the following hold: 
\begin{enumerate}[$(i)$]
\item Every $5$-dimensional quadratic form over $E(X)$ is isotropic.
\item Quadratic forms over $E(X)$ satisfy the local-global principle for isotropy with respect to $\zz$-valuations which are trivial on $E$.
\end{enumerate}
\end{thmA}

The property that $E$ is square-reflexive can be described by a first-order theory in the language of fields (\Cref{R:srf-1st-order}). 
Hence, by Theorem~A, the combination of $(i)$ and $(ii)$ 
can be expressed for the base field $E$ by a first-order theory in the language of fields.
We do not know whether $(ii)$ alone 
can be expressed by a first-order theory.
For $(i)$ this is the case, by \cite[Theorem~1]{Pre87}, where a degree bound is given for isotropic vectors of an isotropic quadratic form over $E(X)$ in terms of the dimension and the degrees of the coefficients.

We have two main types of examples of square-reflexive fields:

\begin{thmB}
If $E$ is finite or pseudo-algebraically closed, then $E$ is square-reflexive.
\end{thmB}

When $E$ is finite, we derive this result (\Cref{FFSR}) from Kornblum's theorem.
In \Cref{transfer}, we use quadratic form transfers to construct an absolutely irreducible curve $C$ over $E$ starting from a separable polynomial $f\in E[X]$ and a square-class in $E[X]/(f)$, such that the presence of an $E$-rational point on $C$ means that the square-class is realised by $X-a$ for some $a\in E$ (\Cref{T:transfer-curve}).
If this holds for all square-classes of $E[X]/(f)$ and the affine curve $Y^2=f(X)$ has a point of odd degree, then $f$ is square-reflexive (\Cref{P:linear-square-class-realisation}).
This leads to the statement for pseudo-algebraically closed fields (\Cref{C:PAC}).

Theorem~B could be derived via Theorem~A from the Hasse-Minkowski Theorem for global function fields, respectively from its analogue  due to I.~Efrat \cite{Efr01}
 for function fields of curves over pseudo-algebraically closed fields.
Our direct proofs of Theorem B now yield new proofs for the local-global principle for quadratic forms over $E(X)$ in the two cases (\Cref{C:finite-LGPI} and \Cref{C:PAC}).

In view of Theorem~A, the presence of a square-free polynomial that is not square-reflexive gives rise to a counterexample to the local-global principle for isotropy of $4$-dimensional forms over $E(X)$; see \Cref{L:main}.
Such counterexamples are characterised by a nontrivial unramified element in the Brauer group of a hyperelliptic curve over $E$ (given by the determinant of the form). 
Inspired by this connection,  we explore in \Cref{hyperelliptic} the relation of square-reflexivity of $f\in E[X]$ to properties of the hyperelliptic curve $Y^2=f(X)$ over $E$.

Square-reflexive polynomials may be relevant for the study of quadratic forms over $E(X)$ even when $E$ is not square-reflexive.
We call $E$ \emph{weakly square-reflexive} if in $E[X]$ every square-free polynomial divides a  square-reflexive polynomial.
By \Cref{T:sqf-div-sqref}, this has strong consequences for quadratic forms over $E(X)$:

\begin{thmC}
Assume that $E$ is weakly square-reflexive. Then the following hold:
\begin{enumerate}[$(i)$]
\item Every $5$-dimensional quadratic form over $E(X)$ is isotropic.
\item Every finite number of $3$-dimensional quadratic forms over $E(X)$ of trivial determinant represent a common element of $\mg{E(X)}$.
\end{enumerate}
\end{thmC}

In fact, $(i)$ is a consequence of $(ii)$ in a more general context, by \Cref{P:nr-sl-u}.
Condition $(ii)$ can be reformulated by saying that any finite number of quaternion algebras over $E(X)$ (or symbols in $\k_2E(X)$) have a common slot.

In \Cref{R:wsrf-ex}, we sketch an argument, using methods from arithmetic geometry, showing that the field $\cc(\!(t)\!)$ is weakly square-reflexive.
On the other hand, it was shown recently that the field $E=\cc(t)$ does not satisfy $(ii)$:  in \cite{ChTi19}, four quaternion algebras over $E(X)=\cc(t,X)$ are constructed which have no common slot, that is, the pure parts of their norm forms do not represent a common nonzero element.
Hence, $\cc(t)$ is not weakly square-reflexive; see \Cref{E:wsrf-nex}.

For some fields $E$,
it is unknown whether Condition $(i)$ holds. Condition $(i)$ implies that the cohomological $2$-dimension of $E$ is $0$ or  $1$ (\Cref{L:rat-res-field-u}), but an example constructed in \cite{CTM04} shows that the converse does not hold.
It may therefore be interesting to see whether our methods, and Theorem~C in particular, give a new avenue to this problem.
\begin{qu*}
In which of the following cases is $E$  weakly square-reflexive?
\begin{enumerate}[$(1)$]
\item $E$ is quasi-finite,~i.e.~$E$~has~a~unique field extension of degree $d$ for any~$d\geq 1$.
\item $E$ is quadratically closed, i.e.~$\mg{E}=\sq{E}$.
\item Specifically, $E=\qq_{\mathsf{quad}}$, the quadratic closure of $\qq$.
\end{enumerate}
\end{qu*}

\subsection*{Acknowledgments}
We wish to express our gratitude to Jean-Louis Colliot-Th\'el\`ene,
 Nicolas Daans, Ido Efrat,  Arno Fehm, David Grimm and David Leep for inspiring discussions, questions and suggestions related to this research.
This article is based on Parul Gupta's PhD-thesis prepared under the supervision of Karim Johannes Becher (\emph{Universiteit Antwerpen}) and Arno Fehm (\emph{Technische Universit\"at Dresden}) in the framework of a joint PhD at \emph{Universiteit Antwerpen} and \emph{Universit\"at Konstanz}.

\section{Preliminaries}\label{Preliminaries}

Our general references for 
quadratic form theory over fields are \cite{EKM}, \cite{Lam05} and \cite{Scharlau}.
In particular, we refer to these sources for the  
definition of the dimension, the orthogonal sum $\perp$, the tensor product $\otimes$ and the relation of isometry $\simeq$ for quadratic forms.
By a \emph{quadratic form}, or simply a \emph{form}, we shall always mean a regular quadratic form.

Let $F$ always be a field of characteristic different from $2$.
For $a_1, \ldots, a_n \in \mg F$, the diagonal quadratic form $a_1X_1^2+\dots+a_nX_n^2$ is denoted by $\la a_1, \ldots,a_n\ra$. 
For $a,b\in\mg{F}$ we denote  by $\lla a,b\rra$ the quadratic form $\la 1,-a,-b,ab\ra$ over $F$. 
Forms of this shape are called $2$-fold Pfister forms. 

Let $\vf$ be a quadratic form over $F$. We denote  its dimension by $\dim(\vf)$. 
For $n=\dim(\vf)$, we can find $a_1,\dots,a_n\in\mg F$ such that $\vf\simeq\la a_1,\ldots,a_n\ra$, and then the \emph{determinant of $\vf$} can be given as the class  $(a_1\cdots a_n)\sq{F}$  in the group $\scg{F}$.
Any quadratic form isometric to $\lambda \vf$ for some $\lambda\in\mg{F}$ is said to be \emph{similar to $\vf$}.

The form $\varphi$ is called \emph{isotropic} if it represents $0$ over $F$ nontrivially, that is, if $\varphi(v)=0$ for some $v\in F^n\setminus\{0\}$, otherwise it is called \emph{anisotropic}. 
The $2$-dimensional quadratic form $\la 1,-1\ra$ is denoted by $\hh$ and called the \emph{hyperbolic plane}.
If $\varphi\simeq r\times \hh$ for some $r\in\nat$, where $r\times \hh$ denotes the $r$-fold orthogonal sum $\hh\perp\ldots\perp \hh$, then the form $\varphi$ is called \emph{hyperbolic}.

Let $\k_2F$ denote the abelian group generated by \emph{symbols}, which are elements of the form $\{a,b\}$ with $a,b\in\mg{F}$, subject to the defining relations that the pairing $$\{\cdot,\cdot\}: \mg{F}\times\mg{F}\lra \k_2 F$$ is bilinear, $\{a,1-a\}=0$ for all $a\in\mg{F}\setminus\{1\}$ and $\{a^2,b\}=0$ for all $a,b\in\mg{F}$. 
It follows from the defining relations that $\k_2F$ is $2$-torsion and that, for any $a,b\in\mg{F}$, we have $\{a,b\}=\{b,a\}$, and furthermore $\{a,b\}=\{a+b,-ab\}$ when $a+b\neq 0$.
Let $\k_1F$ denote the square class group $\scg{F}$ in additive notation, where we write $\{a\}\in \k_1F$ for the element given by the square class of $a\in\mg{F}$, thus having the relation $\{a\}+\{b\}=\{ab\}$ for all $a,b\in\mg{F}$.

Let $\xi \in \k_2F$. An element $a \in \mg F$ is called a  \emph{slot of $\xi$} if  $\xi = \{a, b\}$ for some  $b\in \mg F$. 
Given a subset $S\subseteq \k_2F$, an element $a\in\mg F$ is a \emph{common slot for $S$} if $a$ is a slot of every element of $S$. 
We say that $\k_2F$ is \emph{strongly linked} if every finite subset of $\k_2F$ has a common slot.

Here is a criterion for an element of $\mg{F}$ to be a slot of a given symbol in $\k_2F$.

\begin{prop}\label{firstslot}
For $a,b,c \in \mg{F}$ the following are equivalent:
\begin{enumerate}[$(i)$]
\item $\la c, -a, -b, ab\ra$ is isotropic over $F$.
\item $\lla a, b\rra$ is isotropic over $F(\sqrt{c})$.
\item $\lla a,b \rra \simeq \lla c,d \rra$ for some $d \in \mg F$.
\item $\{a, b\} = \{c, d\}$ for some $d \in \mg F$. 
\end{enumerate}
\end{prop}

\begin{proof}
See \cite[Lemma~2.14.2]{Scharlau} for the equivalence of $(i)$ and $(ii)$.
The equivalence of $(ii)$ and $(iii)$ is a special case of \cite[Lemma~6.11]{EKM}.
See  \cite[Lemma~5.2]{EKM} for the equivalence of $(iii)$ and $(iv)$.
\end{proof}


\begin{cor}\label{P:symbol-zero}
For $a,b\in\mg{F}$
the following are equivalent:
\begin{enumerate}[$(i)$]
\item $\la -a, -b, ab\ra$ is isotropic over $F$.
\item $\lla a,b\rra$ is isotropic over $F$.
\item $\lla a,b\rra$ is hyperbolic over $F$.
\item $\{a, b\} = 0$ in $\k_2F$. 
\end{enumerate}
\end{cor}
\begin{proof}
For the equivalence of $(i)$ and $(iii)$, see~\cite[Corollary~2.11.10]{Scharlau}.
The equivalence of $(ii)$, $(iii)$ and $(iv)$ follows from \Cref{firstslot} by taking $c=1$.
\end{proof}

The \emph{$u$-invariant} of $F$ is defined as 
$$u(F)= \sup\,\{\dim (\varphi)\mid \varphi \text{ anisotropic quadratic form over } F\}\,\,\in \,\,\nat\cup\{\infty\}\,.$$

\begin{cor}\label{u2}
We have  $u(F) \leq 2$ if and only if $ \k_2F =0$. 
\end{cor}
\begin{proof}
Note that $u(F)\leq 2$ if and only if every $3$-dimensional form over $F$ is isotropic.
As every $3$-dimensional  form over $F$ is similar to $\la -a,-b,ab\ra$ for some $a,b\in\mg{F}$, the statement thus follows from \Cref{P:symbol-zero}.
\end{proof}

The field $F$ is \emph{real} if it admits a field ordering, and \emph{nonreal} otherwise.
By the Artin-Schreier Theorem (see \cite[Chap.~VIII, Theorem 1.10]{Lam05}), $F$ is real if and only if $-1$ is not a sum of squares in $F$.
Note that, if $F$ is real, then $u(F)=\infty$ because the form $\la 1,\dots,1\ra$ is anisotropic in any dimension. 

\begin{prop}\label{P:nr-sl-u}
If $F$ is nonreal and any three symbols in $\k_2F$ have a common slot, then $u(F)\leq 4$.
\end{prop}
\begin{proof}
See \cite[Corollary 5.7]{Btriple}. 
\end{proof}

For a finite field extension $F'/F$, the norm map $\N_{F'/F}:F'\rightarrow F$ induces a group homomorphism
$\N_{F'/F}:\k_1F' \rightarrow \k_1 F$ (we use for both maps the same notation).

\begin{lem}\label{L:k1norm-surj}
Assume that $u(F')\leq 2$ holds for every finite field extension $F'/F$.
Let $L/F$ be a finite field extension.
Then, for any $c\in \mg{F}$, there exists $k\in\nat$ and $x\in L$ such that $\N_{L/F}(x)=c^{2k+1}$.
In particular, $\N_{L/F}:\k_1L\to \k_1F$ is surjective.
\end{lem}
\begin{proof}
For any finite field extension $F'/F$, the hypothesis implies that the norm map $\N_{F''/F'}:F''\to F'$ is surjective for any quadratic field extension $F''/F'$.
In view of the composition rule for computing norms in successive finite extensions, the same conclusion holds more generally when $F''/F'$ is a finite $2$-extension, that is, if there exist $r\in\nat$ and a sequence of intermediate fields $(F_i)_{i=0}^r$ where $F_i/F_{i-1}$ is a quadratic extension for $1\leq i\leq r$ and such that $F_0=F'$ and $F_r=F''$.

Let $M/F$ be the normal closure of $L/E$. Since $\car(F)\neq 2$, it follows by Galois theory that there is a subfield $K$ of $M$ containing $F$ such that $M/K$ is a finite $2$-extension  and $[K:F]=2k+1$ for some $k\in\nat$. 
By the previous observation, $\N_{M/K}:L\to K$ is surjective. 
Hence, for $c\in\mg{F}$, there exists $x\in L$ with $\N_{M/K}(x)=c$, and letting $y=\N_{M/L}(x)$,
we conclude that 
$$\N_{L/F}(y)=\N_{M/F}(x)=\N_{K/F}(c)=c^{2k+1}\,,$$
and in particular $\{c\}=\{c^{2k+1}\}=\{\N_{L/F}(y)\}=\N_{L/F}(\{y\})$ in $\k_1F$.
\end{proof}

\begin{cor}\label{P:2quasifinite-k1normsurj}
Assume that $|\scg{F'}|=2$ for every finite field extension $F'/F$.
Then $\N_{F'/F}:\k_1F'\to \k_1F$ is an isomorphism for every finite field extension $F'/F$.
\end{cor}
\begin{proof}
For any $a\in\mg{F}\setminus\sq{F}$, since $|\scg{F(\sqrt{a})}|=2=|\scg{F}|$, it follows by \cite[Chap.~VII, Theorem 3.8]{Lam05} that $\N_{F(\sqrt{a})/F}:F(\sqrt{a})\to F$ is surjective. 
This readily implies that $u(F)\leq 2$.
In view of the hypothesis, the same argument applies to any finite field extension $F'/F$, whereby also $u(F')\leq  2$.
Now \Cref{L:k1norm-surj} shows that, for any finite field extension $F'/F$, the homomorphism $\N_{F'/F}:\k_1F'\to\k_1F$ is surjective, and hence it is an isomorphism, because $|\k_1F'|=|\scg{F'}|=2=|\scg{F}|=|\k_1F|$.
\end{proof}

\begin{ex}
The hypothesis of \Cref{P:2quasifinite-k1normsurj} obviously holds when $F$ is a finite field of odd cardinality.
In this case, for a finite field extension $F'/F$, the norm map $\N_{F'/F}:F'\to F$ is given by $x\mapsto x^N$ where
 $N$ denotes the index of the subgroup $\mg{F}$ in the cyclic group $\mg{F'}$, which also shows that $\N_{F'/F}$ is surjective.
 \end{ex}

Let $v$ be a $\zz$-valuation on $F$.
We denote by $\mc{O}_v$ the associated valuation ring of $F$, by $\mfm_v$ its maximal ideal and by $\kappa_v$ its residue field $\mc{O}_v/\mfm_v$. 
For $a\in \mc{O}_v$ we write $\ovl{a}$ for the residue class $a+\mfm_v$ in $\kappa_v$.
The valuation $v$ is called \emph{nondyadic} if $v(2)=0$ and \emph{dyadic} otherwise.
We call $v$ \emph{complete} if its valuation ring $\mc{O}_v$ is complete with respect to the $\mfm_v$-adic topology. The completion of $F$ with respect to $v$, denoted by $F_v$, is the fraction field of the $\mfm_v$-adic completion of the valuation ring $\mc{O}_v$, and it becomes naturally a field extension of $F$ such that $v$ extends uniquely to a $\zz$-valuation on $F_v$, which has the same residue field $\kappa_v$.

By \cite[Lemma~2.1]{Milnor}, for $n\geq 2$, there is a unique homomorphism 
$$\partial_v:\k_2F\lra \k_{1}\kappa_v$$ such that $\partial_v(\{f,g\})=v(f)\cdot\{ \ovl{g}\}\,\,\,\mbox{ in }\,\k_1\kappa_v$ holds for all $f\in\mg{F}$ and $g\in\mg{\mc{O}}_v$.
For $f,g\in\mg{F}$ we obtain that $f^{-v(g)}g^{v(f)}\in\mg{\mc{O}}_v$ and 
\begin{eqnarray*}
\partial_v(\{f,g\}) & = & \{(-1)^{v(f)v(g)}\ovl{f^{-v(g)}g^{v(f)}}\}\,\,\,\mbox{ in }\, \k_1\kappa_v\,.
\end{eqnarray*}

In the context of local-global principles for quadratic forms with respect to valuation completions, we use Springer's theorem to check the local conditions.

\begin{prop}[Springer]\label{Springer}
Let $v$ be a complete nondyadic $\zz$-valuation on $F$. 
\begin{enumerate}[$(a)$]
\item Given $r\in\nat$ and $a_1,\dots,a_r\in\mg{\mc{O}}_v$, the form $\la a_1,\dots,a_r\ra$ over $F$ is isotropic if and only if the form $\la \ovl{a}_1,\dots,\ovl{a}_r\ra$ over $\kappa_v$  is isotropic.
\item 
For $r,s\in\nat$, $a_1,\dots,a_r,c_1,\dots,c_s\in\mg{\mc{O}}_v$ and $t\in\mg{F}$ such that $v(t)$ is odd,
the form $\la a_1,\dots,a_r\ra\perp t\la c_1,\dots,c_s\ra$ over $F$ is isotropic if and only if one of the forms $\la \ovl{a}_1,\dots,\ovl{a}_r\ra$ and $\la \ovl{c}_1,\dots,\ovl{c}_s\ra$ over $\kappa_v$
 is isotropic.
\end{enumerate}

\end{prop}

\begin{proof}
This is a reformulation of \cite[Chap.~VI, Proposition 1.9]{Lam05}.
\end{proof}

\begin{lem}\label{L:Springer}
Let $v$ be a complete nondyadic $\zz$-valuation on $F$. Then we have $u(F)=2u(\kappa_v)$. 
Moreover, for $n\in\nat$ and $a_1,\dots,a_n\in\mg{F}$ such that the  form $\la a_1,\dots,a_n\ra$ over $F$ is anisotropic, the following hold:
\begin{enumerate}[$(a)$]
\item If $v(a_1)\equiv v(a_2)\equiv \dots\equiv v(a_n)\bmod 2$, then $n\leq u(\kappa_v)$.
\item If $n=2u(\kappa_v)$, then $v(a_1\cdots a_n)\equiv u(\kappa_v)\bmod 2$.
\end{enumerate}
\end{lem}
\begin{proof}
See \cite[Chap.~VI, Corollary 1.10]{Lam05} for the equality $u(F)=2u(\kappa_v)$.

$(a)$\,
Assume that $v(a_1)\equiv \dots\equiv v(a_n)\bmod 2$. Then $\la a_1,\dots,a_n\ra$ is similar to  $\la c_1,\dots,c_n\ra$ for certain $c_1,\ldots,c_n\in\mg{\mc{O}}_v$.
It follows that $\la c_1,\dots,c_n\ra$ is anisotropic.
Hence $\la \ovl{c}_1,\dots,\ovl{c}_n\ra$ is anisotropic over $\kappa_v$, by \Cref{Springer}, whereby $n\leq u(\kappa_v)$.

$(b)$\,
We may assume that $v(a_1),\dots,v(a_r)$ are even and $v(a_{r+1}),\dots,v(a_n)$ are odd, for some $r\leq n$.
Then $v(a_1\cdots a_n)=v(a_1)+\dots +v(a_n)\equiv n-r\bmod 2$.
By $(a)$, we have $r,n-r\leq u(\kappa_v)$, so if $n=2 u(\kappa_v)$, then $r=n-r=u(\kappa_v)$.
\end{proof}

\begin{cor}\label{four-dimensionlocal}
Let $v$ be a complete nondyadic $\zz$-valuation on $F$ with $u(\kappa_v)\leq 2$. 
Let $d\in\mg{F}$ be such that $v(d)$ is odd.
Then every $4$-dimensional form over $F$ of determinant $d\sq{F}$ is isotropic.
\end{cor}

\begin{proof}
Since $u(F)=2 u(\kappa_v)$, the statement is obvious if $u(\kappa_v)=1$.
If $u(\kappa_v)=2$, then the statement follows by part $(b)$ of \Cref{L:Springer}. 
\end{proof}

\begin{prop}\label{localstronglinkage}
Let $v$ be a complete nondyadic $\zz$-valuation on $F$. Assume that $u(\kappa_v)\leq 2$. 
Then the following hold:
\begin{enumerate}[$(a)$]
\item $\{a,b\}=0$ in $\k_2F$ for every $a,b\in\mg{\mc{O}}_v$.
\item For any $\pi \in \mg F$ such that $v(\pi)$ is odd, every element $\k_2F$ is equal to $\{\pi, u\}$ for some $u \in \mg {\mc O_v\!\!}$.
In particular, $\k_2F$ is strongly linked. 
\item $\partial_v :\k_2F \rightarrow \k_1\kappa_v$ is an isomorphism.
\end{enumerate}
\end{prop}

\begin{proof}
For $a,b\in\mg{\mc{O}}_v$,
as $u(\kappa_v)\leq 2$, it follows by \Cref{L:Springer} that $\la -a,-b,ab\ra$ over $F$ is isotropic, and hence $\{a,b\}=0$ in $\k_2E$, by \Cref{P:symbol-zero}. This shows~$(a)$.

To show $(b)$ and $(c)$, we fix $\pi\in\mg{F}$ such that $v(\pi)$ is odd.
By the bilinearity of the pairing $\{\cdot,\cdot\}: \mg{F}\times\mg{F}\lra \k_2 F$ and the fact that $F$ is the fraction field of $\mc{O}_v$,
it follows that $\k_2 F$ is generated by the symbols $\{a,b\}$ and $\{\pi,u\}$ with $a,b,u\in\mg{\mc{O}}_v$. 
Hence $(b)$ follows directly from $(a)$.

For $u\in\mg{\mc{O}_v\!\!}$, we have $\partial_v(\{\pi,u\})=\{\ovl{u}\}$, and if $\{\ovl{u}\}=0$, then $\ovl{u}\in\sq{\kappa_v\!\!}$, and as $v$ is complete,  we obtain by Hensel's Lemma that $u\in\sq{F}$, whereby $\{\pi,u\}=0$. In view of $(b)$, this shows that the homomorphism $\partial_v:\k_2F \rightarrow \k_1\kappa_v$ is injective, and as it surjective in general, it is
an isomorphism.
\end{proof}

\begin{lem}\label{L:rat-res-field-u}
Assume that $u(F)<8$.
Let $v$ be a nondyadic $\zz$-valuation on $F$.
Then $u(\kappa_v)\leq 2$ and $\k_2\kappa_v=0$.
\end{lem}
\begin{proof}
In view of \Cref{u2} it suffices to show that $u(\kappa_v)\leq 2$.
Assume on the contrary that $u(\kappa_v)>2$.
Then there exists an anisotropic $3$-dimensional quadratic form over $\kappa_v$, and after scaling we may choose it to be of the form $\la -\ovl{a\vphantom{b}},-\ovl{b},\ovl{ab}\ra$ for certain $a,b\in\mg{\mc{O}_v\!\!}$.
It follows by \Cref{P:symbol-zero} that the $2$-fold Pfister form $\lla \ovl{a\vphantom{b}},\ovl{b}\rra$ over $\kappa_v$ is anisotropic.
We choose $\pi\in\mc{O}_v$ with $v(\pi)=1$.
We obtain by \cite[Lemma 19.5]{EKM} that the $8$-dimensional 
form $\lla a,b\rra \perp \pi\lla a,b\rra$ over $F$ is anisotropic, in contradiction to the hypothesis that $u(F)<8$.
\end{proof}

\begin{lem}\label{LGPlocalslot}
Let $v$ be a nondyadic $\zz$-valuation on $F$.  
Let $d\in \mg{\mc O}_v$ and $b,c\in \mg{F}$ be such that $\la d, -b, -c, bc \ra$ is isotropic over $F$. 
Then either $\partial_v( \{b,c\})=0 $ or $\partial_v( \{b,c\}) =\{ \ovl {d}\}$ in $\k_1\kappa_v$.   
\end{lem}

\begin{proof}
Assume that $\partial_v( \{b,c\}) \neq 0$ in $\k_1\kappa_v$. 
Since the symbol $\{b,c\}$ remains unchanged if we scale $b$ or $c$ by an element of $\sq{F}$, we may assume that each of $v(b)$ and $v(c)$ is either $0$ or $1$. 
However, if $v(b)=v(c)=0$, then $\partial_v( \{b,c\}) =0$, contrary to the assumption.
Using that $\{b,-b^{-1}c\}=\{b,c\}=\{c,b\}=\{c,-c^{-1}b\}$, we
may therefore reduce to the case where $v(b) =1$ and $v(c)=0$. Hence $c \in \mg {\mc O}_v$, and it follows that $0\neq \partial_v(\{b,c\}) = \{ \ovl{c}\}$ in $\k_1\kappa_v$, whereby $\ovl{c} \notin \sq{\kappa}_v$.  
Since $v(b)=1$,  $c,d\in\mg{\mc{O}}_v$ and the form
$\la d, -b, -c, bc \ra = \la d, -c \ra \perp -b ~\la 1, -c \ra$ is isotropic over $F_v$,
it follows by \Cref{Springer} that one of the residues $ \ovl{dc}$ and $\ovl{c}$ in $\kappa_v$ is a square. 
As $\ovl{c}\notin\sq{\kappa}_v$, we conclude that $\ovl{dc} \in \sq{\kappa}_v$, whereby $\partial_v(\{b,c\}) = \{\ovl{c}\} = \{\ovl{d}\}$ in $\k_1\kappa_v$.  
\end{proof}

\section{Ramification sequences}\label{RS} 

Let $\Omega_F$ denote the set of all $\zz$-valuations on $F$. 
For a subfield $E$ of $F$ we put $$\Omega_{F/E}=\{v\in\Omega_F\mid v(\mg{E})=0\}\,.$$
For $\Omega\subseteq \Omega_F$ and $n\in\nat$, we say that \emph{$n$-dimensional quadratic forms over $F$ satisfy the local-global principle for isotropy with respect to $\Omega$} if every anisotropic quadratic form of dimension $n$ over $F$ remains anisotropic over $F_v$ (the completion of $F$ with respect to $v$) for some $v
\in\Omega$.

We  turn  to the situation  where $F$ is the function field of the projective line $\mathbb{P}^1_E$ over a field $E$.
By the choice of a generator, we identify $F$ with $E(X)$, the rational function field in one variable $X$ over $E$.
We denote by $\deg(f)$ the degree of $f\in E[X]$, with the convention that $\deg(0)=-\infty$.
For $f\in E[X]\setminus\{0\}$, we set $$E_f = E[X]/(f)\,,$$
which is a commutative $E$-algebra with $\dim_E(E_f)=\deg(f)$.

We denote by $\mc{P}$ the set of monic irreducible polynomials in $E[X]$.
Any $p\in\mc{P}$ determines a $\zz$-valuation $v_p$ on $E(X)$ which is trivial on $E$ and with $v_p(p)=1$. There is further a unique $\zz$-valuation $v_\infty$ on $E(X)$ such that $v_\infty(f)=-\deg(f)$ for all $f\in E[X]$.
We set $\mc{P}'=\mc{P}\cup\{\infty\}$. 
By \cite[Theorem~2.1.4]{EP}, we have
$$\Omega_{E(X)/E}=\{v_p\mid p\in\mc{P}'\}\,.$$

\begin{prop}\label{P:123-lgp}
Quadratic forms of dimension at most $3$ over $E(X)$ satisfy the local-global principle for isotropy with respect to $\Omega_{E(X)/E}$.
\end{prop}
\begin{proof}
Set $F=E(X)$. Consider an anisotropic quadratic form $\vf$ of dimension $n\leq 3$ over $F$.
We claim  that $\vf$ is anisotropic over $F_v$ for some $v\in\Omega_{F/E}$.
When $n\leq 1$, this is obvious, since $n$-dimensional forms are anisotropic and $\Omega_{F/E}\neq \emptyset$.

Assume next that $n=2$. Then $\vf$ is similar to $\la 1,-f\ra$ for a square-free polynomial $f\in E[X]$. 
If $f\in \mg{E}$, then since $\vf$ is anisotropic, we have $f\notin \sq{E}$, whence $\vf$ is anisotropic over $E(\!(X)\!)=F_{v_X}$.
If $f\notin E$, then we take a factor  $p\in\mc{P}$ of $f$ and, as $v_p(f)=1$, we find that $f\notin \sq{F_{v_p}\!\!\!}$, whereby $\vf$ is anisotropic over $F_{v_p}$.

Assume finally that $n=3$. Then $\vf$ is similar to $\la 1,-f,-g\ra$ for certain square-free polynomials $f,g\in E[X]$.
We show the claim by induction on $\deg(f)+\deg(g)$.
If $\deg(f)=\deg(g)=0$, then $\vf$ is defined over $E$ and therefore remains anisotropic over $E(\!(X)\!)=F_{v_X}$.
Assume now that $\deg(f)+\deg(g)\geq 1$. We may suppose that $\deg(f)\geq \deg(g)$.
Hence $f\notin E$. 
Suppose first that $g+(f)$ is a non-square in $E_f$.
Since $f$ is square-free, we find a factor $p\in\mc{P}$ of $f$ such that $\ovl{g}=g+(p)$ is a non-square in ${E}_p$.
Then $v_p(f)=1$, $v_p(g)=0$ and $\la 1,-\ovl{g}\ra$ is anisotropic over $E_p$, the residue field of $v_p$. 
It follows by 
\Cref{Springer} that $\la 1,-f,-g\ra$ is anisotropic over $F_{v_p}$, and hence so is $\vf$.
Suppose now that $g+(f)$ is a square in $E_f$.
Hence there exists $h\in E[X]$ with $\deg(h)<\deg(f)$ such that $g\equiv h^2\bmod f$. 
Then $h^2-g=ff^\ast$ for some $f^\ast\in E[X]\setminus\{0\}$.
Thus $ff^\ast$ is represented by $\la 1,-g\ra$, whereby $\la 1,-g\ra\simeq ff^\ast\la 1,-g\ra$.
Scaling by $ff^\ast$ we obtain that $\vf$ is similar to $\la 1,-f^\ast,-g\ra$.
Note that $\deg(f^\ast)<\deg(f)$. Hence, after replacing $f^\ast$ by a square-free polynomial in $E[X]$ in the same class of $\scg{F}$, the induction hypothesis applies and yields that $\vf$ is anisotropic over $F_v$ for some $v\in\Omega_{F/E}$.
\end{proof}

Let $f\in E[X]\setminus\{0\}$. 
We denote by $\supp(f)$ the set of monic irreducible factors of $f$ in $E[X]$ and by $\lc(f)$ the leading coefficient of $f$. 
For $\alpha\in E_f$ and $p\in\supp(f)$, we denote by $\alpha_p$ the image of $\alpha$ under the natural surjection $E_f\to E_p$.

For $p \in \mc P$, the residue field of the $\zz$-valuation $v_p$ on $F=E(X)$ is naturally isomorphic to $E_p=E[X]/(p)$, so we identify it henceforth with $E_p$.
We further denote the residue field of $v_\infty$ by $E_\infty$ and identify it with $E$.
We set $$\mathfrak{R}_2'(E)=\bigoplus_{p\in\mc{P}'} \k_1E_p\,.$$
For $\rho=(\rho_p)_{p\in \mc{P}'}\in{\mathfrak{R}}'_2(E)$, we define the \emph{support of $\rho$} to be the finite set $$\supp(\rho)=\{p\in\mc{P}'\mid \rho_p\neq 0\}\,.$$ 
For $p\in\mc P'$, we denote 
 the homomorphism $\partial_{v_p}:\k_2E(X)\to \k_1E_p$ by  $\partial_p$.
We obtain the \emph{ramification map} 
$$\partial=\bigoplus_{p\in\mc{P}'} \partial_p\,:\,\k_2E(X) \lra \mathfrak{R}_2'(E)\,.$$
For $\alpha\in\k_2 E(X)$ we call $\partial(\alpha)\in\mf R'_2(E)$ the \emph{ramification of $\alpha$}.

Summation over $\N_{E_p/E}:\k_1 E_p\to \k_1 E$ for all $p\in\mc{P}'$ yields a homomorphism
$\N:{\mathfrak{R}}'_2(E) \to \k_1E\,.$
We denote by $\mathfrak{R}_2(E)$ the kernel of $\N:{\mathfrak{R}}'_2(E) \lra \k_1E$.
By \cite[(7.2.1), (7.2.4) and (7.2.5)]{GS},
 we obtain an exact sequence
\begin{equation}\label{MES}
0 \lra \k_2E \lra \k_2E(X) \stackrel{\partial}\lra \mathfrak{R}'_{2}(E) \stackrel{\N}\lra \k_1E \lra 0 \,.
\end{equation}
In particular, $\mathfrak{R}_2(E)$ is equal to the image of ${\partial}:\k_2E(X)\lra \mathfrak{R}'_{2}(E)$.
The elements of $\mf R_2(E)$ are  called \emph{ramification sequences}.

Assume that $f\in E[X]\setminus\{0\}$ is square-free. Then $f=\lc(f)\cdot p_1\cdots p_r$ for some $r\in\nat$ and distinct $p_1,\dots,p_r\in\mc{P}$, and the Chinese Remainder Theorem yields a natural isomorphism of $E$-algebras
$E_f\to \prod_{i=1}^r E_{p_i}$,
which induces a surjective group homomorphism $\mg{E}_f\to \bigoplus_{i=1}^r \k_1 E_{p_i}$.
We obtain a commutative diagram:
\begin{equation}
\begin{tikzcd}\label{E:norm-diagram}
E_f^{\times} \arrow[r, "\mathsf{N}_{E_f/E}"] \arrow[d]               & E^{\times} \arrow[d] \\
\bigoplus_{p\in \mathcal P} \mathsf \k_1 E_p \arrow[r, "\mathsf{N}"'] & \mathsf{k}_1E       
\end{tikzcd}
\end{equation}

The following result is partially contained in \cite[Proposition 3.1]{BR13}.
\begin{lem}\label{L:norm-at-infinity}
Let $\rho\in\mf R_2(E)$.
Let $f\in E[X]\setminus\{0\}$ be square-free and 
such that $\supp(\rho)\subseteq \supp(f)\cup\{\infty\}$.
Then, for any $\alpha\in \mg{E}_f$ such that $\rho_p=\{\alpha_p\}$ for every $p\in\supp(f)$, we have $\rho_\infty=\{\N_{E_f/E}(\alpha)\}$ in $\k_1E$.
In particular, if $\deg(f)=1$ and $\rho_\infty=0$, then $\rho=0$.
Furthermore, if $\supp(\rho)\subseteq \{\infty\}$, then $\rho=0$.
\end{lem}
\begin{proof}
As $\N(\rho)=0$ and $E_\infty=E$, we have $\rho_\infty=
\N_{E_\infty/E}(\rho_\infty)=\sum_{p\in\supp(f)} \N_{E_p/E}(\rho_p)$ in $\k_1E_\infty=\k_1E$.
The first part of the statement now follows by the commutativity of the diagram \eqref{E:norm-diagram}.
For any $q\in\mc{P}$ with $\deg(q)=1$, we have $E_q=E$, hence if $\supp(\rho)\subseteq\supp(q)$, then $\rho_q=\N_{E_q/E}(\rho_q)=\N_{E_\infty/E}(\rho_\infty)=\rho_\infty=0$ in $\k_1E_q=\k_1E$ and it follows that $\rho=0$.
Furthermore, if $\rho_p=0$ for all $p\in\mc{P}$, then $\rho_\infty=\N_{E_\infty/E}(\rho_\infty)=\sum_{p\in\mc{P}}(\rho_p)=0$ in $\k_1E$.
\end{proof}

The following statement is well-known, see e.g.~\cite[Chap.~IX, Theorem~4.6]{Lam05} for the case where $E$ is finite or~\cite[Theorem 6.1]{Sch72} for a more general statement.

\begin{thm}[Hilbert Reciprocity]\label{HR} 
Assume that $|\scg{E'}|=2$ for every finite field extension $E'/E$. 
Let $\rho\in\mf R'_2(E)$. 
Then $\rho\in\mf R_2(E)$ if and only if $|\supp(\rho)|$ is even.
\end{thm}
\begin{proof}
By \Cref{P:2quasifinite-k1normsurj}, for  $p\in\mc{P}'$, $\N_{E_p/E}:\k_1E_p\to \k_1E$ is an isomorphism of groups of order two.
Hence $\N(\rho)=0$ in $\k_1E$ if and only if $|\supp(\rho)|$ is even. 
\end{proof}

To motivate our discussion of a special property for square-free polynomials, we recall the following fact.

\begin{prop}\label{P:ram-reduct}
Let $f\in E[X]$ be square-free.
Let $\rho\in\mathfrak{R}_2(E)$ be such that $\supp(\rho)\subseteq \supp(f)\cup\{\infty\}$.
Then there exists $g\in E[X]$ coprime to $f$ such that $\deg(g)<\deg(f)$ and $\supp(\rho-\partial(\{f,g\}))\subseteq \supp(g)\cup\{\infty\}$.
\end{prop}
\begin{proof}
Let $\alpha\in\mg{E_f\!\!}$ be such that $\rho_p=\{\alpha_p\}$ for every $p\in\supp(f)$.
We take $g\in E[X]$ with $\deg(g)<\deg(f)$ such that 
$g+(f)=\alpha$ in $E_f$. 
Then $g$ is coprime to $f$ and $\supp(\rho-\partial(\{f,g\}))\subseteq \supp(fg)\cup\{\infty\}$.
Since $\rho_p=\partial_p(\{f,g\})$ for all $p\in\supp(f)$, we obtain that $\supp(\rho-\partial(\{f,g\}))\subseteq \supp(g)\cup\{\infty\}$.
\end{proof}

This statement readily yields by induction that any $\rho\in\mf R_2(E)$ is the image under $\partial$ of a sum of $\lfloor\frac{\deg(\rho)}2\rfloor$ symbols in $\k_2E(X)$, where $\deg(\rho)=\sum_{p\in\supp(\rho)}[E_p:E]$; see \cite[Theorem 3.10]{BR13}.
One can ask whether any $\rho\in\mf R_2(E)$ can be realised as the ramification of a single symbol, or more specifically, whether for any square-free polynomial $f\in E[X]$ such that $\supp(\rho)\subseteq\supp(f)\cup\{\infty\}$ there exists a polynomial $g\in E[X]$ with $\supp(\rho-\partial(\{f,g\}))=\emptyset$, hence such that
 $$\rho=\partial(\{f,g\}).$$

\begin{prop}\label{degreebound}
Let $f\in E[X]$ be square-free.
Let $\rho\in\mathfrak{R}_2(E)$ be such that $\supp(\rho)\subseteq \supp(f)\cup\{\infty\}$.
Let $\alpha\in \mg{E_f\!\!}$ be such that $\rho_p=\{\alpha_p\}$ for all $p\in\supp(f)$.
The following two conditions are equivalent:
\begin{enumerate}[$(i)$]
\item There exists $g\in E[X]\setminus\{0\}$ such that $\partial(\{f,g\})=\rho$.
\item There exists $g'\in E[X]$ coprime to $f$ with ${g'}+(f)\in\alpha\sq{E_f\!\!}$  and $f+(g')\in\sq{E_{g'}\!\!\!}$.
\end{enumerate}
Moreover, if $(i)$ and $(ii)$ hold, then $g,g'\in E[X]$ in $(i)$ and $(ii)$ can be chosen with $\deg(g)\leq \frac{1}2\deg(f)$ and $\deg(g')\leq \frac{3}2\deg(f)$.
\end{prop}
\begin{proof}
If $(ii)$ is satisfied with $g'\in E[X]$,  then it follows  that $\partial(\{f,g'\})=\rho$, hence one can choose $g=g'$ to satisfy $(i)$.
Therefore $(ii)$ implies $(i)$.

Assume that $(i)$ holds.
We fix a polynomial $g\in E[X]\setminus\{0\}$ of minimal degree such that $\partial(\{f,g\})=\rho$.
Then $g$ is square-free. 
Since $\supp(\rho)\subseteq \supp(f)\cup\{\infty\}$, $f$ is a square modulo $g$.
Hence, we can write $f = Q^2 - g G$
with $Q, G \in E[X]$ such that $\deg(Q) <\deg(g)$.
Then $\{f,g \} = \{f,G\}$ in $\k_2E(X)$, whereby $\deg(G) \geq \deg(g)$ in view of the choice of $g$. It follows that $\deg(f) = \deg( gG) \geq 2 \deg(g)$, whereby $\deg(g) \leq \hbox{$\frac{1}{2}$}\deg(f)$.

Let $h$ be the greatest common monic divisor of $f$ and $g$ in $E[X]$. Write $f= f_0 h$ and $g = g_0 h$ with $f_0, g_0\in E[X]$. 
Let further $r,g'\in E[X]$ be such that $g'$ is monic and square-free and
 $g_0 (h-f_0) =g'r^2$. 
Since $f$ and $g$ are square-free, these choices imply that $f$ and $g'$ are coprime and that for every $p\in \supp(f_0)$ we have $\rho_p =\{\ovl{g\vphantom{h}}\} =\{\ovl{g_0 h}\} =\{\ovl{g'}\}$ in $\k_1E_p$.
For every $p\in \supp(h)$ we have $\rho_p =\{-\ovl{g_0 f_0}\}=\{\ovl{g'}\} $ in $\k_1E_p$. 
Hence, $\{\alpha_p\} =  \rho_p = \{\ovl{g'}\}$ for every $p \in\supp(f)$.
As $f$ is square-free we conclude that $g'+(f) \in \alpha\sq{E_f\!\!}$.
 
Since $\rho = \partial(\{f, g\})$ and $\supp(\rho) \subseteq \supp(f)\cup \{\infty\}$, we obtain that $f$ is a square modulo~$g_0$. 
Obviously $f=f_0h$ is also a square modulo $h-f_0$. 
Since $g'$ is square-free and divides $g_0(h-f_0)$, it follows that $f$ is a square modulo~$g'$.   

Note that 
$\deg(g_0) \leq \deg(g) \leq \hbox{$\frac{1}{2}$}\deg(f)$ and $\deg(h-f_0) \leq \deg(f)$.
We conclude that 
$\deg(g') \leq \deg (g_0 (h-f_0)) \leq \hbox{$\frac{3}{2}$}\deg(f)$.
Hence we have proven that $(i)$ implies $(ii)$, as well as the degree bounds on $g$ and $g'$ at the end of the statement.
\end{proof}

The idea for the degree bounds in \Cref{degreebound} stems from \cite[Proposition~1.4]{Sivatski} and was pointed out to us by M.~Raczek.
\smallskip

We call a square-free polynomial $f\in E[X]\setminus\{0\}$ \emph{square-reflexive} if, for every $\alpha \in \mg{E_f\!\!}$ with $\mathsf{N}_{E_f/E}(\alpha) \in \sq E \cup \lc(f)\sq E$, there exists a polynomial $g\in E[X]$ coprime to $f$ 
such that 
\begin{center}
$g+(f)\in\alpha\sq{E_f\!\!}$\qquad and
\qquad $f+(g)\in\sq{E_g\!\!}$.
\end{center}
Note that a square-reflexive polynomial is square-free by definition.

\begin{rem}\label{remSRuglycond}
If $f\in E[X]$ is square-reflexive, then by \Cref{degreebound} and \Cref{L:norm-at-infinity}, for every $\alpha\in \mg{E_f\!\!}$  satisfying the {norm condition} 
$$\mathsf{N}_{E_f/E}(\alpha) \in \sq E \cup \lc(f)\sq E\,,$$ there exists a square-free polynomial $g\in E[X]$
such that $f$ is a square modulo $g$ and $\partial_p(\{f,g\})=\{\alpha_p\}$ in $\k_1E_p$ for all $p\in \supp(f)$.
When $f$ has odd degree, then one can conclude the same for all $\alpha\in \mg{E_f\!\!}$, regardless of the norm condition.

Suppose now that $f\in E[X]$ is square-reflexive of even degree. Set $c=\lc(f)$.
Consider $\alpha\in \mg{E_f\!\!}$ and assume that Condition $(ii)$ in \Cref{degreebound} is satisfied, that is, 
there exists $g\in E[X]$ coprime to $f$ such that $g\in\alpha\sq{E_f\!\!}$ and $f$ is a square modulo~$g$.
Then \Cref{L:norm-at-infinity} applied to $\rho=\partial(\{f,g\})$ yields that $\N_{E_f/E}(\alpha)\in\sq{E}\cup c\sq{E}$. Hence, in this case the condition on the norm is necessary for the existence of  $g$.
\end{rem}

\begin{prop}\label{SRtoRC}
Let $f\in E[X]$ be square-free and $c=\lc(f)$. The following are equivalent:
\begin{enumerate}[$(i)$]
\item $f$ is square-reflexive.
\item For every $\rho \in \mf R_2(E)$, such that $\supp(\rho) \subseteq\supp(f)\cup\{\infty\}$ and either $\rho_\infty =0$ or $\rho_\infty =\{c\}$  in $\k_1E$, there exists $g\in E[X]\setminus \{0\}$ 
such that $\rho = \partial(\{f,g\})$.
\item For every $\alpha \in \mg E_f$  with $\N_{E_f/E}(\alpha) \in \sq{E} \cup c\sq{E}$, there exists $g' \in E[X]$ coprime to $f$ and with $\deg(g') \leq\frac{ 3}{2}\deg(f)$ such that
$g'+(f)\in\alpha\sq{E_f\!\!}$ and
 $f+(g')\in\sq{E_{g'}\!\!\!}$.
\end{enumerate}
If there exists $s\in E[X]$ such that $f-s^2$ has a factor of odd degree, then $(i)$---$(iii)$ are further equivalent to the following:
\begin{enumerate}[$(i')$]\setcounter{enumi}{1}
\item For any $\rho\in\mf R_2(E)$ such that $\supp(\rho)\subseteq\supp(f)$ there exists $g\in E[X]\setminus \{0\}$ such that $\rho=\partial(\{f,g\})$.
\end{enumerate} 
\end{prop}

\begin{proof}
In view of \Cref{L:norm-at-infinity} it follows directly from \Cref{degreebound} that Conditions $(i)$---$(iii)$ are equivalent.
Trivially, $(ii)$ implies $(ii')$.

Assume now that $s\in E[X]$ is such that $f-s^2$ has a factor $r\in\mc{P}$ of odd degree and that $(ii')$ holds.
We claim that this implies $(ii)$.
Consider $\rho \in \mf R_2(E)$ such that  $\supp(\rho) \subseteq \supp(f)\cup \{\infty\}$ and either $\rho_\infty =0$ or $\rho_\infty =\{c\}$  in $\k_1E$.  
If $\rho_\infty =0$, then $\supp(\rho)\subseteq\supp(f)$ and hence $\rho=\partial(\{f,g\})$ for some $g\in E[X]\setminus\{0\}$, by $(ii')$.
Assume now that $\rho_{\infty} = \{c\} \neq 0$ in $\k_1E$. 
Consider the ramification sequence
$\rho' = \rho -\partial ( \{f, r\})$.
Note that $\partial_{\infty}(\{f,r\})= \{c\}=\rho_\infty$ in $\k_1E$ and $f$ is a square modulo $r$.
Hence  $\supp(\rho') \subseteq \supp(f)$ and it follows by $(ii')$ that there exists $g \in E[X]\setminus\{0\}$ such that $\partial(\{f ,g\}) = \rho'$. Thus
$\rho = \partial (\{f, rg \})$.
\end{proof}

\begin{rem}\label{C:sqref-1storder}
Let $d\in\nat$.
The condition on $E$ that every square-free polynomial in $E[X]$ of degree $d$ is square-reflexive can be expressed by a first-order sentence in the language of fields.  
This follows from the degree bound contained in the characterization $(iii)$ of \Cref{SRtoRC}.
\end{rem}

\section{Square-reflexive polynomials over different fields}\label{SRP-examples}

In this section we study the property for polynomials to be square-reflexive by looking at examples and counterexamples.

\begin{prop}\label{SRuptodeg2}
Let $f\in E[X]$ be square-free with $\deg(f)\leq 2$.
Then $f$ is square-reflexive.
\end{prop}

\begin{proof}
Set $c=\lc(f)$.
We shall verify Condition $(ii)$ of \Cref{SRtoRC}.
Consider any $\rho\in\mf R_2(E)$ such that $\supp(\rho)\subseteq \supp(f)\cup\{\infty\}$ and either $\rho_\infty=0$ or $\rho_\infty=\{c\}$. We need to show that $\rho=\partial(\{f,g\})$ for some $g\in E[X]\setminus\{0\}$.

By \Cref{P:ram-reduct}, there exists $g\in E[X]\setminus\{0\}$ with $\deg(g)<\deg(f)$ and 
$\supp(\rho-\partial(\{f,g\}))\subseteq \supp(g)\cup\{\infty\}$.
If $g\in\mg{E}$, then $\supp(\rho-\partial(\{f,g\}))\subseteq \{\infty\}$, and we conclude by \Cref{L:norm-at-infinity} that  $\rho=\partial(\{f,g\})$.

We may now assume that $g\notin \mg{E}$. Then $\deg(g)=1$ and $\deg(f)=2$.
It follows that $\partial_\infty(\{f,g\})=\{c\}$.
If $\rho_\infty=\{c\}$, then
$\supp(\rho-\partial(\{f,g\}))\subseteq \supp(g)$, and as 
$\deg(g)=1$ we conclude by \Cref{L:norm-at-infinity} that $\rho=\partial(\{f,g\})$.

We may from now on assume that $\rho_\infty=0$. 
Suppose first that $f=cp_1p_2$ for certain $p_1,p_2\in\mc{P}$ with $\deg(p_1)=\deg(p_2)=1$.
As $E_1=E$, we may choose $c'\in\mg{E}$ such that $\rho_{p_1}=\{c'\}$.
Since $\partial_{p_1}(\{f,c'\})=\{c'\}=\rho_{p_1}$ and $\rho_\infty=0=\partial_\infty(\{f,c'\})$, it follows that $\supp(\rho-\partial(\{f,c'\}))\subseteq\supp(p_2)$.
As $\rho-\partial(\{f,c'\})\in \mf R_2(E)$ and $\deg(p_2)=1$, we 
obtain by \Cref{L:norm-at-infinity} that $\rho=\partial(\{f,c'\})$, so we replace $g$ by~$c'$.
Suppose finally that $f$ is irreducible.
Let $\alpha=g+(f)\in E_f$. 
Since $\rho_\infty=0$, \Cref{L:norm-at-infinity} yields that $\N_{E_f/E}(\alpha)\in\sq{E}$.
By \cite[Chap.~VII, Remark 3.9]{Lam05}, this implies that $\alpha\in \mg{E}\cdot\sq{E_f\!\!}$. Now we can replace $g$ by an element of $\mg{E}$ and reduce to the case considered at the beginning.
\end{proof}

The condition on $E$ that all square-free polynomials in $E[X]$ of a certain degree $d$ are square-reflexive becomes restrictive when $d$ increases.
We give a quite general example of a monic square-free cubic polynomial that is not square-reflexive.

\begin{prop}\label{P:sqref-deg3}
Let $v$ be a valuation on $E$ such that $v(\sq{E})\neq v(\mg{E})$.
Then the polynomial $X(X+1)(X+a)$ is not square-reflexive over $E$ for any $a\in \mg{\mc{O}_v\!\!}\setminus\{1\}$.
\end{prop}

\begin{proof}
We set $f = X(X+1)(X+a)$ and fix $\pi \in \mg E$ with $v(\pi)\notin v(\sq{E})$.
In particular, $\pi \notin \sq E$.
Consider the ramification sequence $\rho \in \mf R_2(E)$ 
given by $\rho_X=\{\pi\}$, $\rho_{X+1}=\{\pi\}$ and $\rho_p=0$ for all other $p\in\mc{P}'$.
In particular $\rho_{X+a}=0$, $\rho_\infty=0$ and $\supp(\rho)\subseteq\supp(f)$.

Suppose now that $f$ is square-reflexive.
Then, by \Cref{degreebound}, there exists a polynomial $g\in E[X]\setminus\{0\}$ with $\deg(g)\leq 1$ and such that $\rho=\partial(\{f,g\})$.

Since $\rho_\infty =0$ and $\rho\neq 0$, we must have $g\notin\mg{E}$.
Hence $\deg(g)=1$ and $0=\rho_\infty=\partial_\infty(\{f,g\})=\{-\lc(g)\}$, whereby
 $\lc(g)\in-\sq{E}$. 
Hence we may rescale $g$ by an element of $\sq{E}$ to achieve that $g =c-X$ for some $c \in E$.
Hence $\rho=\partial(\{f,c-X\})$.
Note that $\partial_X(\{f,-X\})=\{a\}\neq \{\pi\}=\rho_X$, because $a\pi\notin\sq{E}$.
Hence $c\neq 0$.
We observe that $\{c\}=\partial_X(\{f,c-X\})=\rho_X=\{\pi\}$ as well as $\{c+1\}=\partial_{X+1}(\{f,c-X\})=\rho_{X+1}=\{\pi\}$.
Therefore $c\pi,(c+1)\pi\in\sq{E}$.
As $v(\pi)\notin v(\sq{E})$, we get that $v(c),v(c+1)\notin v(\sq{E})$.
In particular $v(c)<0=v(a)$ and thus $v(c+a)=v(c)\notin v(\sq{E})$, whereas $\{c+a\}=\partial_{X+a}(\{f,c-X\})=\rho_{X+a}=0$.
This is a contradiction.
\end{proof}

\begin{prop}\label{P:real-srp}
Assume that $E$ is a real field.
Then any square-reflexive polynomial in $E[X]$ has at most three roots in $E$.
\end{prop}
\begin{proof}
In view of the hypothesis, we may fix a field ordering $<$ on $E$.
Consider a square-free polynomial $f\in E[X]$ with roots $a_1,a_2,a_3,a_4\in E$, numbered such that 
$a_1< a_2< a_3< a_4$. 
We set $h=(X-a_1)(X-a_3)$ and consider the symbol $\sigma=\{-1,h\}$.
Note that $\supp(\partial(\sigma))=\supp(h)\subseteq\supp(f)$.

To show that $f$ is not square-reflexive, we show that $\partial(\sigma)\neq\partial(\{f,g\})$ for all $g\in E[X]\setminus\{0\}$.
Suppose on the contrary that $\partial(\sigma)=\partial(\{f,g\})$ for some $g\in E[X]\setminus\{0\}$.
By the exact sequence (\ref{MES}), we obtain that $\sigma\equiv \{f,g\}\bmod \k_2E$.
This remains true after replacing the field $E$ by its real closure with respect to the fixed ordering.
Hence we assume now that $E$ is real closed.
Then $\k_2E$ is a group of order $2$ generated by the symbol $\{-1,-1\}$.
Since $\{-1,h\}\equiv \{f,g\}\bmod \k_2E$, we obtain that 
$\{f,g\}=\{-1,\varepsilon h\}$ where $\varepsilon=\pm 1$.
By \Cref{firstslot} it follows that the quadratic form $\la f,1,-\varepsilon h,-\varepsilon h\ra$ over $E(X)$ is isotropic.
Hence there exist rational functions $f_1,f_2,f_3\in E(X)$ such that $-f=f_1^2-\varepsilon h(f_2^2+f_3^2)$.
If $\varepsilon=1$, then we obtain that $f(x)\leq 0$ for any $x\in E$ with $a_1\leq x\leq a_3$, which contradicts the fact that $f$ has a simple root in $a_2$.
If $\varepsilon=-1$, then we get that $f(x)\leq 0$ for any $x\in E$ with $a_3<x$, contradicting the fact that $f$ has a simple root in $a_4$. 
\end{proof}

\begin{prop}\label{L:sqr-deg4-k2triv}
Assume that in $E[X]$ every monic square-free polynomial of degree $4$ divides a square-reflexive polynomial.
Then $u(E)\leq 2$.
\end{prop}
\begin{proof}
It follows by \Cref{P:real-srp} that $E$ is nonreal.
Suppose that $u(E)>2$. Then $\k_2 E\neq 0$, by \Cref{u2}.
Hence there exists a nonzero symbol $\sigma$ in $\k_2E$.
By \cite[Lemma~4.3]{Bec20}, we can write 
$\sigma=\{a,b\}$ 
with elements $a,b\in\mg{E}$ such that  $a,ab,(a-4)b\notin\sq{E}$. 
We now consider 
$$\rho = \partial(\{X^2+(a+1)X+a,a\}+\{X^2 +aX+a,ab\})\,.$$
Note that the polynomial $(X^2+(a+1)X+a)(X^2 +aX+a)\in E[X]$ is square-free.
Hence, by the hypothesis it divides a square-reflexive polynomial $f\in E[X]$.
As $\supp(\rho)\subseteq\supp(f)$, we obtain that
$\rho= \partial(\{f,g\})$ for some $g\in E[X]$.
By \cite[Theorem~4.2]{Bec20}, this is in contradiction with having that $\{a,b\}=\sigma\neq 0$.
\end{proof}

\section{Square-reflexive and weakly square-reflexive fields}\label{SRF}

We call the field $E$ \emph{square-reflexive} if every square-free polynomial in $E[X]$ is square-reflexive.
We further call $E$ \emph{weakly square-reflexive} if every square-free polynomial in $E[X]$ divides a square-reflexive polynomial in $E[X]$.

\begin{rem}\label{R:srf-1st-order}
It follows from \Cref{C:sqref-1storder} that square-reflexivity as a property for fields of characteristic different from $2$ is described by a first-order theory in the language of fields.
Note that this property is preserved under direct limits.
\end{rem}

\begin{prop}\label{P:wsr-field}
If $E$ is square-reflexive, then every valuation on $E$ has a $2$-divisible value group.
\end{prop}
\begin{proof}
This follows from  \Cref{P:sqref-deg3}.
\end{proof}

We will now show that finite fields of odd cardinality are square-reflexive.
We obtain this as a consequence of the following classical analogue to Dirichlet's theorem on primes in an arithmetic progression.

\begin{thm}[Kornblum]\label{T:Kornblum}
Assume that $E$ is a finite field.
Let  $f,g\in E[X]$ be coprime. Then there exists $N\in\nat$ such that, for every $n\in\nat$ with $n\geq N$, there exists $q\in\mc{P}$ with $\deg(q)=n$
and $q\equiv g\bmod f$.
\end{thm}
\begin{proof}
This version follows from \cite[Theorem 4.8]{rosen}.
\end{proof}

\begin{thm}\label{FFSR}
Every finite field of odd cardinality is square-reflexive.
\end{thm}

\begin{proof}
Assume that $E$ is finite.
Let $f\in E[X]$ be square-free and $c=\lc(f)$.
To show that $f$ is square-reflexive we verify Condition $(ii)$ in \Cref{SRtoRC}.

Consider a ramification sequence $\rho\in\mf R_2(E)$ with $\supp(\rho)\subseteq\supp(f)\cup\{\infty\}$ and either $\rho_\infty=0$ or $\rho_\infty=\{c\}$.
Let $\alpha\in \mg{E_f\!\!}$ be such that $\rho_p=\{\alpha_p\}$ for every $p\in\supp(f)$.
Let $c'\in\mg{E}$ be such that $\rho_\infty=\{c'\}$ in $\k_1E$.
\Cref{T:Kornblum} implies that there exists $q\in\mc{P}$ with  $\alpha=q+(f)$ and such that $q$ is of even degree if $c'\in\sq{E}$ and of odd degree otherwise.
If $c'\in\sq{E}$, then since $\deg(q)$ is even we obtain that $\partial_\infty(\{f,q\})=0=\rho_\infty$.
If $c'\notin\sq{E}$, then $c'\sq{E}=c\sq{E}$, and since $\deg(q)$ is odd we obtain that $\partial_\infty(\{f,q\})=\{c\}=\rho_\infty$. 
Note further that $\rho_p=\partial_p(\{f,q\})$ for all $p\in\mc{P}\setminus\{q\}$.
For $\rho'=\rho-\partial(\{f,q\})$, we obtain that $\supp(\rho')\subseteq\supp(q)$. 
But \Cref{HR} states that $|\supp(\rho')|$ is even. We conclude that $\supp(\rho')=\emptyset$. Hence $\rho'=0$ and therefore $\rho=\partial(\{f,q\})$.
\end{proof}

In \Cref{C:PAC} we will obtain an analogue of \Cref{FFSR} when $E$ is a pseudo-algebraically closed field.
Other examples of square-reflexive fields are sparse.

\begin{qu}\label{Q:qc-srf}
Assume that $\mg{E}=\sq{E}$. Does it follow that $E$ is square-reflexive?
\end{qu}

Note that a positive answer would have strong consequences,  in particular it would follow (by \Cref{T:sqf-div-sqref}  below) that $u(E(X))\leq 4$; this is not known to be true in general when $\mg{E}=\sq{E}$.
If the answer to \Cref{Q:qc-srf} turns out to be negative in general, it would be still interesting to see whether it is positive when $E$ is the quadratic closure of $\qq$.

We can at least give a nontrivial family of square-reflexive polynomials over a field satisfying the hypothesis in \Cref{Q:qc-srf}.
\begin{prop}\label{QCSR}
Assume that $\mg{E}=\sq{E}$. 
Let $f\in E[X]$ be square-free. Assume that $f=pq$ where $p,q \in E[X]$ are such that $\deg(p)\leq 5$ and $q$ is split.
Then $f$ is square-reflexive.
\end{prop}

\begin{proof}
We may suppose that $q$ contains all linear factors of $f$.
As $\mg{E}=\sq{E}$, quadratic polynomials in $E[X]$ are reducible, and $\mg{E_g\!\!}=\sq{E_g\!\!}$ holds for every separable split polynomial $g\in E[X]$. 
So, either $p=1$, or $3\leq \deg(p)\leq 5$ and $p$ is irreducible.
Consider $\alpha \in \mg {E}_f$. It follows by \cite[Lemma~2.5]{Bec02} that $\{\alpha_p\}=\{\ovl{g}\}$ in $\k_1E_p$ for a square-free polynomial $g\in E[X]\setminus (p)$ whose irreducible factors are of degree at most $2$. Since quadratic polynomials in $E[X]$ are reducible, it follows that $g$ is split. Hence $\mg{E_g\!\!} =\sq {E_g\!\!}$ and $\mg{E_q\!\!} =\sq {E_q\!\!}$. We conclude that $\alpha\ovl{g} \in \sq {E}_{f}$ and that $f$ is a square modulo $g$. 
\end{proof}

\begin{thm}\label{T:sqf-div-sqref}
Assume that $E$ is weakly square-reflexive.
Then $u(E(X))\leq 4$ and $\k_2E(X)$ is strongly linked.
\end{thm}
\begin{proof}
By   \Cref{L:sqr-deg4-k2triv}, the hypothesis implies that $u(E)\leq 2$.
Hence $E$ is nonreal and $\k_2 E=0$, by \Cref{u2}.
It therefore suffices to show that $\k_2E(X)$ is strongly linked, because then $u(E(X))\leq 4$ follows by \Cref{P:nr-sl-u}.

Consider an arbitrary finite subset $\mc{S}$ of $\k_2E(X)$. 
If $E$ is infinite, then using a variable transformation, we may assume that $\partial_\infty(\alpha)=0$ for all $\alpha\in\mc{S}$, and since $E$ is weakly square-reflexive, there exists a square-reflexive polynomial $f\in E[X]$ such that 
$\supp(\alpha)\subseteq \supp(f)$ for all $\alpha\in\mc{S}$.
If $E$ is finite, then $\mg{E}=\sq{E}\cup c\sq{E}$ for some $c\in\mg{E}$, and we choose a square-free polynomial $f\in E[X]$ with $\lc(f)=c$ and such that $\supp(\alpha)\subseteq \supp(f)\cup\{\infty\}$ for all $\alpha\in\mc{S}$, and we have that $f$ is square-reflexive, by \Cref{FFSR}.

Since $\k_2 E=0$, the homomorphism
$\partial:\k_2E(X)\to \mf R_2(E)$ occurring in the exact sequence \eqref{MES} is injective. 
Therefore, and by Condition $(ii)$ in \Cref{SRtoRC}, the choice of $f$ implies that every element of $\mc{S}$ is of the form $\{f,g\}$ for some $g\in E[X]\setminus\{0\}$.
\end{proof}

The last statement motivates the following question.
\begin{qu}
Assume that $E$ is nonreal and $\k_2E(X)$ is strongly linked. Does it follow that $E$ is weakly square-reflexive?
\end{qu}

\begin{ex}\label{R:wsrf-ex}
Consider the field $E=\cc(\!(t)\!)$. 
Note that $E$ is not square-reflexive, by \Cref{P:wsr-field}, because $E$ carries a discrete valuation.
Using geometric methods going back to \cite{FS89} and \cite{Ford96}, it was shown in \cite{BG21} that $\k_2F$ is strongly linked when $F$ is the function field of a curve over $E$. 
Similar arguments can be used to prove that $E$ is weakly square-reflexive.

Let us sketch the argument showing that $E$ is weakly square-reflexive.
First we need a sufficiently large set of square-reflexive polynomials in $E[X]$: 
Suppose that $f\in E[X]$ is square-free of odd degree and such that there exists 
 a regular flat projective model $\mc{X}$ of $E(X)$ over $T=\cc[\![t]\!]$ such that the principal divisor 
$(f)_{\mc{X}}$ has normal crossings on $\mc{X}$ and each component of the special fiber $\mc{X}_s$ has odd multiplicity  in $(f)_{\mc{X}}$.
Using \cite[Lemma 8.2]{BG21} together with the exact sequence \eqref{MES} and the fact that $\k_2E=0$, it follows from these hypotheses that every element of $\s\in\k_2E(X)$ with $\supp(\partial(\s))\subseteq \supp(f)\cup\{\infty\}$ is of the form $\{f,g\}$ for some $g\in \mg{E(X)}$, whereby $f$ is square-reflexive; to apply \cite[Lemma 8.2]{BG21}  one uses the fact that the unramified Brauer group of any regular flat projective surface over $T$ is trivial.

Using certain facts on divisors from \cite[Section~1]{Sal07},
one can now show that to any normal crossing divisor $D$ on $\mc{X}$ there exists a divisor $D'$ with disjoint support such that $D+D'\equiv (f)_{\mc{X}}\bmod 2$ for some polynomial $f$ as above.
Now, starting with a square-free polynomial $f_0\in E[X]$, one can use embedded resolution of singularities to obtain a model $\mc{X}$ as above such that $(f_0)_{\mc{X}}$ has normal crossings, and applying then the previous step to $D=(f_0)_{\mc{X}}$ yields a square-reflexive polynomial $f\in E[X]$ which is divisible by $f_0$.
\end{ex}

The rational function field over the complex numbers $\cc(Y)$ gives an example of a nonreal field with $u$-invariant $2$ which is not weakly square-reflexive.

\begin{ex}\label{E:wsrf-nex}
We consider the rational function field $F=\cc(X,Y)$. 
In $\k_2F$ we consider the four symbols
$$
\begin{array}{lcllcl}
\s_1 & = & \{Y,X(X+1)\}, &\,\, \s_2 & = & \{Y+1,X(X+1)\} , \\
 \s_3 & = & \{Y,(2X+1)(X+1)\},   &\,\, \s_4 & = & \{Y,(XY+X+1)(X+1)\}\,.
 \end{array}
 $$
Under the substition $x_1=\frac{X}{X+1}$ and $x_2=Y$, we have $F=\cc(x_1,x_2)$ and the four symbols correspond to the four $F$-quaternion algebras 
$$(x_1,x_2)_F\,,\,\, (x_1,x_2+1)_F\,,\,\, (x_1+1,x_2)_F\,,\,\, (x_1x_2+1,x_2)_F\,.$$
It is shown in  \cite{ChTi19} that these four $F$-quaternion algebras do not contain a common maximal subfield.
It follows that $\s_1,\s_2,\s_3,\s_4$ do not have a common slot. In particular $\k_2F$ is not strongly linked.

Consider now the field $E=\cc(Y)$, for which we have $F=E(X)$.
By Tsen's Theorem \cite[Theorem 6.2.8]{GS}, we have $u(E)\leq 2$. Therefore $\k_2E=0$.
Since $\k_2F$ is not strongly linked, $E$ is not  weakly square-reflexive, by \Cref{T:sqf-div-sqref}.
In fact, we claim that the square-free polynomial $$g=X(X+1)(2X+1)(XY+X+1)\in E[X]$$ does not divide any square-reflexive polynomial in $E[X]$.
Suppose on the contrary that there is a square-reflexive polynomial $f\in E[X]$ which is a multiple of $g$.
For $1\leq i\leq 4$, since $\supp(\partial(\s_i))\subseteq \supp(f)$, we get that $\partial(\s_i)=\partial(\{f,h_i\})$ for some $h_i\in E[X]\setminus\{0\}$.
As $\k_2E=0$, it follows 
by the exact sequence \eqref{MES} that $\s_i=\{f,h_i\}$ for $1\leq i\leq 4$. This contradicts the fact that $\s_1,\s_2,\s_3,\s_4$ have no common slot.
\end{ex}

\section{The local-global principle for $4$-dimensional forms} \label{app}

In this section we study the situation when $u(E') \leq 2$ holds for every finite field extension $E'/E$. Under this hypothesis, we give necessary and sufficient conditions on $E$ for having that quadratic forms over $E(X)$ satisfy the local-global principle for isotropy with respect to $\zz$-valuations that are trivial on $E$.

Recall that $$\Omega_{E(X)/E}=\{v_p\mid p\in\mc{P}'\}\,.$$
For $p\in\mc{P}'$ we denote by $E(X)_p$ the completion of $E(X)$ with respect to $v_p$.

\begin{thm}\label{monicLGPI}
Assume that $u(E) \leq 2$. Let $f\in E[X]$ be square-reflexive. Then $4$-dimensional forms over $E(X)$ with determinant $f\sq{E(X)}$ satisfy the local-global principle for isotropy with respect to $\Omega_{E(X)/E}$.  
\end{thm}

\begin{proof}
By \Cref{u2}, the hypothesis implies that $\k_2E=0$.
It follows by the exact sequence (\ref{MES}) that 
$\partial:\k_2E(X)\to \mf R_2(E)$ is injective.

Let $\varphi$ be a $4$-dimensional form over $E(X)$ of determinant $f\sq{E(X)}$. 
Then~$\varphi$ is similar to
$\la f, -g,-h,gh\ra$ 
for certain square-free polynomials $g,h \in E[X]$. Assume that $\varphi$ is isotropic over $E(X)_p$ for all $p \in \mc P'$. 
Set $S = \supp(\partial(\{g,h\}))\setminus \{\infty\}$ and $\mathsf{P}_{S} = \prod_{p\in S}p$. 
Let $c' \in \mg E$ be such that $\partial_\infty (\{g,h\}) = \{c'\}$ in $\k_1E_\infty$.
We set
$$ \alpha  = \left\{\begin{array}{ll}
           \{f, (-1)^{\deg(\mathsf{P}_S)}c'\mathsf{P}_S\} &\mbox{if $\deg(f)$ is odd},\\
           \{f, \mathsf{P}_S\} & \mbox{if $\deg(f)$ is even}.
           \end{array}\right.$$
Now we consider
$$ \rho  = \partial(\alpha -\{g,h\})\,.$$

For any $p \in \mc P\setminus (S \cup \supp(f))$ we have $\rho_p =0 $. 
For any $p \in S \setminus \supp(f)$, it follows by \Cref{LGPlocalslot} that
$$\partial_p ( \{g, h \}) = \{\ovl {f } \} =\partial_p (\alpha) \,,$$
and therefore $\rho_p =0$. 
This shows that 
$$\supp(\rho)\subseteq\supp(f)\cup\{\infty\}\,.$$

Let $c=\lc(f)$. Then 
$$ \partial_{\infty}(\alpha)  = \left\{\begin{array}{ll}
           \{c'c^{\deg(\mathsf{P}_S)}\} &\mbox{if $\deg(f)$ is odd},\\
           \{c^{\deg(\mathsf{P}_S)}\} & \mbox{if $\deg(f)$ is even}.
           \end{array}\right.$$ 
Moreover, if $\deg(f)$ is even and $\partial_\infty(\{g,h\}) \neq 0$, then it follows by \Cref{LGPlocalslot} that $\{c\}=\partial_\infty(\{g,h\}) = \{c'\}$ in $\k_1E_\infty$. 
Therefore in all cases we have either $\rho_\infty = 0$ or $\rho_\infty = \{c\}$ in $\k_1E_\infty$.
 
Since $f$ is square-reflexive, it follows by \Cref{{SRtoRC}} that $\rho=\partial(\{f,q\})$ for some polynomial $q \in E[X]\setminus\{0\}$. 
Using that $\partial$ is injective, we conclude that
$$\{g,h\} = \left\{\begin{array}{ll}
           \{f, (-1)^{\deg(\mathsf{P}_S)}c'q\mathsf{P}_S\} &\mbox{if $\deg(f)$ is odd},\\
           \{f, q \mathsf{P}_S\} & \mbox{if $\deg(f)$ is even}.
           \end{array}\right.$$
Hence $f$ is a slot of $\{g,h\}$. 
We conclude by \Cref{firstslot} that $\varphi$ is isotropic.
\end{proof}

\begin{rem}\label{R:refine}
The proof of \Cref{monicLGPI} shows the following (under the same assumptions on $E$):
If $f\in E[X]$ is square-reflexive, $g,h\in E[X]\setminus\{0\}$ and the quadratic form 
$\vf=\la f,-g,-h,gh\ra$ is anisotropic over $E(X)$, then $\vf$ is anisotropic over $E(X)_p$ for some $p\in\supp(\partial(\{g,h\}))$. 
\end{rem}

\Cref{monicLGPI} has the following remarkable consequences.

\begin{cor}\label{C:finite-LGPI}
Assume that $E$ is a finite field of odd cardinality.
Then $4$-dimen\-sional quadratic forms over $E(X)$ satisfy the local-global principle for isotropy with respect to $\Omega_{E(X)/E}$.
\end{cor}
\begin{proof}
By \Cref{FFSR} square-free polynomials in $E[X]$ are square-reflexive. 
As the determinant of every $4$-dimensional form over $F$ is given by the square-class of a square-free polynomial in $E[X]$, the statement follows by \Cref{monicLGPI}.
\end{proof}

The conclusion in \Cref{monicLGPI} does not hold if we assume that $f\in E[X]$ is square-reflexive but drop the conditions on $E$.
We give two examples where the determinant is given by a separable quadratic polynomial. Recall that square-free quadratic polynomials are square-reflexive, by \Cref{SRuptodeg2}.

\begin{exs}
\begin{enumerate}[$(1)$]
\item E.~Witt showed in \cite{Witt34} that $-1$ is not a sum of two squares in $F'=\qq(X)(\sqrt{-(X^2+21)})$, whereas it is a sum of two squares in every completion of $F'$ with respect to a real or a non-archimedian place.
It follows that the $2$-fold Pfister form $\lla -1,-1\rra$ over $F'$ is anisotropic but becomes isotropic over $F'_{v}$ for every $v\in\Omega_{F'}$.
Using \Cref{firstslot}, it follows that the form $\la -(X^2+21),1,1,1\ra$ over $F=\qq(X)$ is anisotropic and that it becomes isotropic over 
$F_v$ for every $v\in\Omega_{F}$.
\item Using the same arguments based on the idea in \cite{Witt34}, one can show that the quadratic form $\la 5X^2+13, \sqrt{-1},5,5\sqrt{-1}\ra$ over $F=\qq(\sqrt{-1})(X)$ is anisotropic and it becomes isotropic over $F_v$ for every $v\in\Omega_{F}$. Note that $u(\qq(\sqrt{-1}))=4$.
\end{enumerate}
\end{exs}

\Cref{monicLGPI} can be used to show that certain polynomials are not square-reflexive.

\begin{ex}\label{E:cubic}
Let $E=\cc(\!(t)\!)$.
It follows by \Cref{P:sqref-deg3} that the square-free polynomial 
$$f=X(X+1)(X+t)\in E[X]$$
is not square-reflexive.
We now give an alternative argument for this fact, which uses \Cref{monicLGPI}.
Consider the quadratic form 
$$\varphi= \la t,tX,X+1,X+t\ra$$ 
over the field $F=E(X)$.
Note that $\vf$ has determinant $f\sq{F}$ and $F$ is a subfield of $\cc(X)(\!(t)\!)$.
Using Springer's Theorem (\Cref{Springer}), it is easy to see that $\varphi$ becomes isotropic over $F_v$ for every $v\in\Omega_{F/E}$, while it is anisotropic over $\cc(X)(\!(t)\!)$, and hence also over $F$.
Hence the local-global principle for isotropy with respect to $\Omega_{F/E}$ for $4$-dimensional quadratic forms of determinant $f\sq{F}$ is not satisfied. 
By \Cref{monicLGPI}, this implies that $f$ is not square-reflexive over~$E$.
(Note however that quadratic forms over $F$ do satisfy the local-global principle for isotropy with respect to $\Omega_F$, the set of all $\zz$-valuations on $F$, by \cite[Theorem~3.1]{CTPS}.)
\end{ex}

We do not know whether the converse to the implication in \Cref{monicLGPI} holds for an individual square-free polynomial:
 
\begin{qu}
Assume that $u(E)\leq 2$ and let $f\in E[X]$ be such that every anisotropic $4$-dimensional form of determinant $f\sq{E(X)}$ over $E(X)$ remains anisotropic over $E(X)_v$ for some $v\in\Omega_{E(X)/E}$. Does it follow that $f$ is square-reflexive?
\end{qu}

We can take this conclusion if we assume that all $4$-dimensional quadratic forms over $E(X)$ satisfy the local-gobal principle for isotropy with respect to $\Omega_{E(X)/E}$.
To show this is the purpose of the rest of this section.

The following lemma is a variation of \cite[Proposition 3.4]{BR}. 

\begin{lem} \label{bezutian}
Let $f \in E[X]$ be monic square-free and let $g\in E[X]$ be coprime to $f$. Assume that $\deg(f)\geq u(E)$. Then there exists a monic square-free polynomial $g^\ast \in E[X]$ such that $\deg(g^\ast)<\deg(f)$, $\deg(g^\ast) \equiv\deg (f)-1\bmod{2}$ and $gg^\ast$ is a square modulo $f$.
\end{lem}

\begin{proof}
Let $n=\deg(f)$.
Recall that $E_f = E[X]/(f)$ and $\dim_E E_f = n$. Let $\theta = X+(f)$ in $E_f$. 
As $g$ is coprime to $f$ we have $g(\theta)\in\mg{E_f\!\!}$.
Consider the $E$-linear map $s_f : E_f \rightarrow E$, given by $s_f(\theta^{n-1}) =1$ and $s_f(\theta^{i}) = 0$ for $0\leq i \leq n-2$. 
Let $\varphi$ be the quadratic form over $E$ given by the quadratic map 
$$ E_f  \rightarrow E, \,\,x \mapsto s_f(g(\theta) x^2)\,.$$
By \cite[Proposition 3.1]{BR}, $\varphi$ is a regular quadratic form. Since $\dim(\varphi) = \deg(f)$ and $\deg(f)\geq u(E)$, we obtain that $\varphi$ represents $1$ over $E$. 
Let $x\in E_f$ be such that $\varphi (x) =1$, whereby $s_f(g(\theta)x^2) =1$.  
Hence there exist $a_0, \ldots, a_{n-2}\in E$ such that $$g(\theta)x^2 =a_0 +\ldots+a_{n-2}\theta^{n-2}+ \theta^{n-1}\,.$$
Set $g'(X) = a_0 +\ldots+a_{n-2}X^{n-2}+ X^{n-1}$. 
Then $g(\theta)g'(\theta) =(g(\theta)x)^2$ in $E_f$, whereby $gg'$ is a square modulo $f$. 
In the factorial domain $E[X]$ we may write $g' = g^\ast h^2$ for some $g^\ast,h \in E[X] $ such that $g^\ast$ is monic and square-free. Then $gg^\ast$ is a square modulo $f$, and furthermore $n-1 =\deg(g')\equiv \deg(g^\ast)\bmod 2$.     
\end{proof}

\begin{lem}\label{L:main}
Assume that $u(E')\leq 2$ for all finite field extensions $E'/E$.
Let $f\in E[X]$ be square-free, not square-reflexive, and of minimal degree for these properties.
Then there exists an anisotropic $4$-dimensional form of determinant $f\sq{E(X)}$ over $E(X)$ which is isotropic over $E(X)_v$ for every $v\in\Omega_{E(X)/E}$.
\end{lem}
\begin{proof}
By \Cref{u2}, the hypothesis implies that $\k_2 E_p=0$ for all $p\in\mc{P}'$.

Let $c=\lc(f)$.
Since $f$ is not square-reflexive, by \Cref{SRuptodeg2}, we have $\deg(f)>2$, and by \Cref{SRtoRC}, there exists  $\rho \in \mf R_2(E)$ with
$\supp (\rho) \subseteq \supp(f)\cup \{\infty\}$, either $\rho_\infty =0$ or  $\rho_\infty =\{c\}$, and
such that $\rho\neq\partial(\{f,g\})$ for every $g\in E[X]\setminus\{0\}$.

Since $\deg(f) \geq 2\geq u(E)$, \Cref{bezutian} yields that there exists a monic square-free polynomial $g\in E[X]$ with $\deg(g) < \deg(f)$ and $\deg(g) \equiv \deg(f)-1\bmod 2$ and with the property that 
$$\rho_p= \{\ovl{g}\} \quad\mbox{ in }\k_1E_p\quad\mbox{ for every }\quad p \in \supp (f)\setminus\supp(g)\,.$$ 
The choice of $g$ implies that $\partial_\infty(\{f,g\})=\{c^{\deg(g)}\}$ in $\k_1 E_\infty$. 
We set 
$$\rho' = \rho - \partial(\{f,g\})\,.$$ 
We obtain that $\supp(\rho') \subseteq \supp(g) \cup \{\infty\}$
and either $\rho_\infty'=0$ or $\rho_\infty'=\{c\}$.
Let $c'\in\sq{E}\cup c\sq{E}$ be such that $\rho_\infty'=\{c'\}$.
Since $\deg(g) < \deg (f)$, the hypothesis on $f$ yields that $c'g$ is square-reflexive.
Hence we may choose $h \in E[X] \setminus \{0\}$ such that $\rho' = \partial (\{c'g,h\})$.
We thus have $$\rho \, = \, \partial(\{f,g\}+\{c'g,h\})\,.$$
We now consider the quadratic form 
$$\varphi = \la f ,-c'g ,-h, c'gh \ra$$
over $E(X)$, whose determinant is $f\sq{E(X)}$. 

Suppose that $\varphi$ is isotropic over $E(X)$.
It follows by \Cref{firstslot} that there exists $h'\in E[X]\setminus\{0\}$ such that $\{c'g,h\}=\{f,h'\}$,
whereby
$$\rho \, = \, \partial(\{f,g\}+\{f,h'\})=\partial(\{f,gh'\})\,,$$
in contradiction to the choice of $\rho$.
Hence  $\varphi$ is anisotropic over $E(X)$.

It remains to show that $\varphi$ is isotropic over $E(X)_p$ for every $p\in \mc P'$. 

Consider first $p \in \mc P\setminus \supp(f)$. 
As $v_p(f) =0 = v_p(c')$ and $u(E_p)\leq 2$, it follows by \Cref{localstronglinkage} that in $\k_2E(X)_p$ we have $\{f,c'\} =0$ and thus $\{f,g\} = \{f,c'g\}$. Hence
$$\partial_p(\{fh, c'g\})=\rho_p  = 0 \mbox{ in $\k_1E_p$}\,.$$ 
Again by \Cref{localstronglinkage} it follows that $ \{ fh, c'g\} =0$ in $\k_2E(X)_p$, and in view of \Cref{firstslot} we obtain that $\la -fh,-c'g,c'fgh\ra$  is isotropic over $E(X)_p$.
Since this form is a subform of $(-c'fg)\varphi$, we conclude that $\varphi$ is isotropic over $E(X)_p$.

Consider next $p\in \supp(f)$. Since $f$ is square-free, we have $v_p(f) =1$, and since $u(E_p) \leq 2$, it follows by \Cref{four-dimensionlocal} that $\varphi$ is isotropic over $E(X)_p$. 

We come to the last and most tricky case where $p =\infty$. Note that by definition $v_\infty(f)=-\deg(f)$.
If $\deg(f)$ is odd, then since $u(E_\infty) \leq 2$ it follows by \Cref{four-dimensionlocal} that $\varphi$ is isotropic over $E(X)_\infty$. Assume now that $\deg(f)$ is even. In that case we have that $\deg(g)$ is odd, by the choice of $g$.

Note that $\partial_\infty (\{c'g,h\}) = \rho'_\infty$.
Suppose first that $\rho_\infty'=0$. 
Then it follows by part $(c)$ of \Cref{localstronglinkage} that $\{c'g,h\} =0$ in $\k_2E(X)_\infty$, and hence by \Cref{P:symbol-zero} that $\la -c'g,-h,c'gh\ra$ is isotropic over $E(X)_\infty$. 
Since $\la -c'g,-h,c'gh\ra$ is a subform of $\varphi$, we obtain that $\varphi$ is isotropic over $E(X)_\infty$.

Suppose now that $\rho_\infty'\neq 0$. 
As $\rho_\infty'=\{c'\}$ and $c'\in\sq{E}\cup c\sq{E}$, it follows that 
$c\notin\sq{E}$, $cc'\in \sq{E}$ and
$0\neq \{c\}=\rho'_\infty =\partial_\infty (\{cg,h\})$.
Recall that $\deg(f)$ is even, $\deg(g)$ is odd and $\lc(f)=c$.
If $\deg(h)$ is even, then we conclude that $\lc(h)\in c\sq{E}$, whereby $\la f,-h\ra$ is isotropic over $E(X)_\infty$.
If $\deg(h)$ is odd, then we conclude that $\lc(gh)\in-\sq{E}$, whereby $\la f , c'gh \ra$ is isotropic $E(X)_\infty$.
In either case we obtain that $\varphi$ is isotropic over $E(X)_\infty$.
\end{proof}

We come to the main result of this section.

\begin{thm}\label{LGPI4}
The following are equivalent:
\begin{enumerate}[$(i)$]
\item $E$ is square-reflexive.
\item $u(E(X))\leq 4$ and 
 quadratic forms over $E(X)$ satisfy the local-global principle for isotropy with respect to $\Omega_{E(X)/E}$.
\end{enumerate}
\end{thm}

\begin{proof} 
Assume $(i)$. Then $u(E(X))\leq 4$, by \Cref{T:sqf-div-sqref}.
Since by the assumption every square-free polynomial is square-reflexive, it follows by \Cref{monicLGPI} that $4$-dimensional quadratic forms over $E(X)$ satisfy the local-global principle for isotropy with respect to $\Omega_{E(X)/E}$.
In view of \Cref{P:123-lgp}, this establishes $(ii)$.

Assume now that $(ii)$ holds. 
By \Cref{L:rat-res-field-u}, we have
$\k_2 E_p=0$ for every $p\in\mc{P}'$. 
Hence it follows by  \Cref{L:main} that every square-free polynomial in $E[X]$ is square-reflexive, which shows $(i)$.
\end{proof}

\begin{cor}\label{C:srf-Springer}
Let $E'/E$ be a field extension which is a direct limit of finite extensions of odd degree.
If $E'$ is square-reflexive, then so is $E$.
\end{cor}
\begin{proof}
We use the equivalence in \Cref{LGPI4} and show that Condition $(ii)$ holds for $E$ by using that it holds for $E'$.
That is, for an arbitrary anisotropic quadratic form $\varphi$ over $E$, we show that $\dim(\vf)\leq 4$ and that there exists some $v\in\Omega_{E(X)/E}$ such that $\varphi$ remains is anisotropic over $E(X)_v$.

For any finite  extension $L/E$ contained in $E'/E$, the hypothesis on $E'/E$ implies that $[L:E]$ is odd, and as $[L(X):E(X)]=[L:E]$, it follows by Springer's Theorem \cite[Chap.~VII, Theorem 2.4]{Lam05} that $\varphi$ stays anisotropic over $L(X)$.
As $E'(X)/E(X)$ is the direct limit of $L(X)/E(X)$ where $L/E$ ranges over the finite subextensions of $E'/E$, we conclude that $\varphi$ remains anisotropic over $E'(X)$.

Since $E'$ is square-reflexive, \Cref{LGPI4} yields that $u(E'(X))\leq 4$, whereby $\dim(\vf)\leq 4$, and further that $\varphi$ is anisotropic over $E'(X)_w$ for some $w\in\Omega_{E'(X)/E'}$.
Since the extension $E'(X)/E(X)$ is algebraic, it follows that $w(\mg{E(X)})$ is a nontrivial subgroup of
$w(\mg{E'(X)})=\zz$. Hence $\mc{O}_w\cap E(X)=\mc{O}_v$ for a $\zz$-valuation $v$ on $E(X)$. 
Then $v\in\Omega_{E(X)/E}$ and the completion $E(X)_v$ is contained in $E'(X)_w$, whereby $\varphi$ is anisotropic over $E(X)_v$.
\end{proof}

One consequence of \Cref{LGPI4} is that its Condition $(ii)$ is stable under direct limits.
The following example shows that the local-global principle for isotropy of $4$-dimensional quadratic forms over $E(X)$ with respect to $\Omega_{E(X)/E}$  as a property of a field $E$ does not behave so nicely on its own, that is, without being combined with the condition that $u(E(X))\leq 4$ as in $(ii)$.

\begin{ex}
Let $p$ be an odd prime number.
Let $E$ be the maximal unramified extension of $\qq_p$.
By \cite[Theorem 12]{Lan52}, $E$ is a $\mc{C}_1$-field.
Hence $E(X)$ is a $\mc{C}_2$-field, so in particular $u(E(X))\leq 4$.
By \Cref{P:wsr-field}, $E$ is not square-reflexive, because it carries a discrete valuation.
Since $E$ is nonreal, it follows by \Cref{LGPI4} that the local-global principle for isotropy of $4$-dimensional forms over $E(X)$ with respect to $\Omega_{E(X)/E}$ fails.

On the other hand, $E$ is an algebraic extension of $\qq_p$, hence a direct limit of finite extensions.
As explained in \cite[Remark 3.8]{CTPS}, 
for any finite field extension $K/\qq_p$, the local-global principle for isotropy of $4$-dimensional forms over $K(X)$ with respect to $\Omega_{K(X)/K}$ does hold.

This fact can also be retrieved using \cite[Theorem~4.1]{Pop88} from the $\Phi$-property introduced in \cite[Definition 3.3]{Pop88}, which holds for any nondyadic local field, as a consequence of a result of Lichtenbaum \cite[Theorem~4]{Lich69}. 
It follows from \cite[Theorem~3.4]{Pop88} that the $\Phi$-property is a first-order property in the language of fields.
We mention that one can see from the example of the field $E$ here that the $\Phi$-property is not preserved under direct limits.
\end{ex}

\section{Relation to hyperelliptic curves}
\label{hyperelliptic}

In this section we consider the case where $u(E_p)\leq 2$ holds for all $p\in\mc{P}$ and 
explore for a nonconstant square-free polynomial $f\in E[X]$ the relationship between the two conditions that $f$ is square-reflexive and that the affine hyperelliptic curve $Y^2=f(X)$ has a point of odd degree.

\begin{lem}\label{L:hyper-odd}
Let $f\in E[X]$ be square-free of even degree with $\lc(f)\notin\sq{E}$.
If the affine curve $Y^2=f(X)$ has a point of odd degree over $E$, then it has a point of odd degree at most equal to $\frac{\deg(f)}{2}$.
\end{lem}
\begin{proof}
Suppose that $f(x)=y^2$ for certain elements $x,y$ contained in an algebraic field extension of $E$ such that $[E[x,y]:E]$ is odd. We may further assume that $[E[x,y]:E]$ is minimal among all such choices of $(x,y)$.

Since $y^2=f(x)\in E[x]$, we have $[E[x,y]:E[x]]\leq 2$, and since $[E[x,y]:E]$ is odd, we conclude that $y\in E[x]$.
Let $p$  be the minimal polynomial of $x$ over $E$.
Then $\deg(p)$ is odd and $(x,y)$ is a point of the curve $Y^2=f(X)$ over $E$ of degree equal to $\deg(p)$. 
Let $s\in E[X]$ with $\deg(s)<\deg(p)$ be such that $y=s(x)$. 
We obtain that $f-s^2=ph$ for some $h\in E[X]$.

Since $\lc(f)\notin\sq{E}$, we have $\deg(ph)=\max(\deg(f),2\deg(s))$, which is even.
Hence $\deg(h)$ is odd. Let $q$ be an irreducible factor of $h$ of odd degree.
Since $q$ divides $f-s^2$,
the curve $Y^2=f(X)$ has a point $(x',y')$ over $E_q$.
Hence the degree $[E[x',y']:E]$ divides $\deg(q)$, so it is odd.
By the choice of $(x,y)$ we have that 
$\deg(h)\geq\deg(q)\geq [E[x',y']:E]\geq [E[x,y]:E]=\deg(p)$.
Hence $2\deg(s)<2\deg(p)\leq \deg(ph)=\deg(f)$.
\end{proof}

\begin{prop}\label{P:hyper-example}
Let $f\in E[X]$ be square-free and let $p$ be a monic irreducible factor of $f$ such that $f(0)p(0)\in\sq{E}$.
Consider the quadratic form 
$$\varphi=\la f,-p,X,-pX\ra$$ over the field $F=E(X)$. Then:
\begin{enumerate}[$(a)$]
\item For any $v\in\Omega_{F/E}$ with $u(\kappa_v)\leq 2$, the form $\varphi$ is isotropic over $F_v$.
\item If $\deg(f)$ is even, $p(0),\lc(f)\notin\sq{E}$ and $\varphi$ is isotropic over $F$, then the affine curve
$Y^2=f(X)$ has a point of odd degree over $E$.
\end{enumerate}
\end{prop}
\begin{proof}
Note that the hypothesis implies that $f(0)\neq 0$, so $X$ does not divide $f$.

$(a)$\,
To show the isotropy of $\varphi$ over a given extension of $F$, it suffices to show that some subform of $\varphi$ becomes isotropic over this extension.
Since $f(0)p(0)\in\sq{E}$, the subform $\la f,-p\ra$ is isotropic over $F_{v_X}$.
Since $p$ is monic, the subform $\la -p,X,-Xp\ra$ is isotropic over $F_{v_\infty}$.
If $v\in\Omega_{F/E}\setminus\{X,\infty\}$, then $\varphi$ has a subform $\la a,b,c\ra$ with
$a,b,c\in\mg{F}$ such that $v(a)\equiv v(b)\equiv v(c)\bmod 2$, and hence, if $u(\kappa_v)\leq 2$, then it follows by 
\Cref{L:Springer} that $\la a,b,c\ra$ is isotropic over $F_v$.

$(b)$\, 
Assume that $\varphi$ is isotropic over $F$.
Then by \Cref{firstslot}, we have $\{p,-X\}=\{f,g\}$ in $\k_2F$ for some $g\in \mg{F}$, and as $F=E(X)$ we can choose $g$ to be square-free polynomial in $E[X]$. 
Set $a=p(0)$. Then the hypothesis implies that $f(0)\in a\sq{E}$.
Suppose now that $\deg(f)$ is even and $a,\lc(f)\notin\sq{E}$.
Then $\partial_X(\{f,g\})=\partial_X(\{p,-X\})=\{a\}\neq 0$, and as $f(0)\neq 0$ it follows that $X$ divides $g$.
We write $g=X\cdot h$ with $h\in E[X]$. 
Since $p$ is monic, we have 
$\partial_\infty(\{f,g\})=\partial_\infty(\{p,-X\})=0$, and
since $\deg(f)$ is even and $\lc(f)\notin\sq{E}$, it follows that $\deg(g)$ is even.
Therefore $\deg(h)$ is odd. Hence $h$ has a monic irreducible factor $q$ of odd degree.
Note that $q\neq X$, because $X\cdot h=g$ is square-free.
Let $\gamma=X+(q)$ in $E_q=E[X]/(q)$.
If $q$ divides $f$, then $f(\gamma)=0$, and thus $(\gamma,0)$ is a point of odd degree of the affine curve $f(X)=Y^2$ over $E$.
Assume that $q$ does not divide $f$.
Then $q\neq p$ and $f(\gamma)\neq 0$, and we obtain that 
$0=\partial_q(\{p,-X\})= \partial_q(\{f,g\})=\{f(\gamma)\}$ in $\k_1 E_q$, whereby $f(\gamma)=\xi^2$ for some $\xi\in \mg{E_q\!\!}$. 
Then $(\gamma,\xi)$ is a point of odd degree of $f(X)=Y^2$ over $E$.
\end{proof}

\begin{cor}
Let $t\in \mg{E}$ and $r\in\nat$.
Assume that there exists a valuation $w$ on $E$ such that $w(t)\notin w(\sq{E})$.
Then the following hold:
\begin{enumerate}[$(a)$]
\item The affine curve $Y^2=t(X^{4r}-t^2)$ has no point of odd degree over $E$.
\item The form $\varphi=\la t(X^{4r}-t^2),t-X^{2r},X,-X(X^{2r}-t)\ra$ over $E(X)$ is anisotropic.
\item For any $v\in\Omega_{E(X)/E}$ with $u(\kappa_v)\leq 2$, the form $\varphi$ is isotropic over $E(X)_v$.
\end{enumerate}
\end{cor}
\begin{proof}
To show $(a)$, we substitute $X$ for $X^r$ and thus reduce to the case where $r=1$.
Hence we consider the affine curve $Y^2=t(X^4-t^2)$ over $E$.
By \Cref{L:hyper-odd}, it suffices to show that this curve has no $E$-rational point.
Consider $x\in E$.
Since $w(t)\notin w(\sq{E})$, we have $w(t^2)\neq w(x^4)$ and
therefore $w(x^4-t^2)=\min(w(t^2),w(x^4))\in w(\sq{E})$.
It follows that $w(t(x^4-t^2))\notin w(\sq{E})$.
Hence the curve $Y^2=t(X^4-t^2)$ has no $E$-rational point.

Having thus established $(a)$, we can apply \Cref{P:hyper-example} with $f=t(X^{4r}-t^2)$ and $p=X^{2r}-t$ to conclude $(b)$ and $(c)$.
\end{proof}

\begin{prop}\label{P:srp-odd-deg-hyperel}
Assume that $u(L)\leq 2$ for every finite field extension $L/E$.
Let $f\in E[X]\setminus E$ be square-reflexive. 
Then the affine curve $Y^2=f(X)$ has a point of odd degree.
\end{prop}
\begin{proof}
If there exists $s\in E[X]$ such that $f-s^2$ has a factor of odd degree, then 
the curve $Y^2=f(X)$ has a point of odd degree.
This applies in particular when $\deg(f)$ is odd.
On the other hand, if $\deg(f)=2r$ and $\lc(f)=a^2$ for some $r\in\nat$ and $a\in\mg{E}$, we get that $\deg(f-((aX+b)X^{r-1})^2)=2r-1$ for every $b\in E$ for which $2ab$ differs form the coefficient of $f$ in degree $2r-1$. 
We may therefore assume that $\deg(f)$ is even and $\lc(f)\notin\sq{E}$.

We fix a factor $q\in\mc{P}$ of $f$ and set $c=\lc(f)$. 
By \Cref{L:k1norm-surj}, since $u(L)\leq 2$ for every finite field extension $L/E$, there exists some $x\in \mg{E_q\!\!}$ 
with $\N_{E_q/E}(x)\in c\sq{E}$.
Letting $\rho_q=\{x\}$, $\rho_\infty=\{c\}$ and $\rho_{r}=0$ for all $r\in\mc{P}\setminus\{q\}$ defines a ramification sequence $\rho\in\mf R_2(E)$  with $\supp(\rho)\subseteq \supp(f)\cup\{\infty\}$ and $\rho_\infty=\{c\}$.
As $f$ is square-reflexive, there exists $g\in E[X]$ coprime to $f$ with $\rho=\partial(\{f,g\})$.
In particular, we have $\{c\}=\rho_\infty=\partial_\infty(\{f,g\})=\{c^{\deg(g)}\}$ in $\k_1E_\infty$.
As $c\notin\sq{E}$, we conclude that $\deg(g)$ is odd. 
We choose a factor $p\in\mc{P}$  of $g$ of odd degree.
Then it follows that $p\notin\supp(\rho)$, whereby $\partial_p(\{f,g\})=\rho_p=0$.
Hence $f$ is a square modulo $p$. Therefore $Y^2=f(X)$ has a point over $E_p$, and hence of odd degree over $E$.
\end{proof}

\begin{cor}
If $E$ is square-reflexive, then every affine hyperelliptic curve over $E$ has a point of odd degree.
\end{cor}
\begin{proof}
By \Cref{LGPI4}, square-reflexivity of $E$ implies that $u(E(X))\leq 4$, which by \Cref{L:rat-res-field-u} implies that
$u(L)\leq 2$ for all finite field extensions $L/E$. Hence the statement follows from \Cref{P:srp-odd-deg-hyperel}.
\end{proof}

\begin{qu}
Assume that every hyperelliptic curve over $E$ has a point of odd degree.
Does it follow that $E$ is square-reflexive?
\end{qu}

\section{The transfer curve of a polynomial}\label{transfer} 

In this section our aim is to show that pseudo-algebraically closed fields have similar properties as finite fields concerning quadratic forms over the rational function field: If $E$ is pseudo-algebraically closed, then $E$ is square-reflexive and quadratic forms over $E(X)$ satisfy the local-global principle for isotropy with respect to $\Omega_{E(X)/E}$.
Our proof relies on an analysis of a certain type of variety arising from transfer maps of a quadratic form over a finite extension of $E$.

We withdraw our previous convention to reserve the terms \emph{form} and \emph{quadratic form} to regular quadratic forms. 

Let $n,r\in\nat$ with $r\leq n$.
For $n\in\nat$, by a \emph{quadratic form in $E[X_1,\dots,X_n]$} we mean a homogeneous polynomial of degree $2$ in $X_1,\dots,X_n$ over $E$.
The \emph{rank} of a quadratic form in $E[X_1,\dots,X_n]$ is  the rank of the corresponding symmetric matrix in $\mathbb{M}_{n}(E)$.
A \emph{diagonal form in $E[X_1,\dots,X_n]$} is an element of the $E$-subspace generated by the monomials $X_1^2,\dots,X_n^2$.

For $i,j\in\nat$ we denote by $\delta_{ij}$ the Kronecker delta with values $0$ and $1$ in $E$.

\begin{prop}\label{P:pencil-rank-geq3-absolirr}
Let $\mc{F}$ be an $r$-dimensional space of diagonal quadratic forms in $E[X_1,\dots,X_n]$.
Let $\mc{V}$ denote the projective variety over $E$ given as the vanishing set of $\mc{F}$ in $\mathbb{P}_E^{n-1}$.
Then $\mc{V}$ has dimension $n-r-1$.
Furthermore $\mc{V}$ is absolutely irreducible if and only if $\mc{F}$ contains no quadratic form of rank $1$ or $2$.
\end{prop}
\begin{proof}
Any $E$-basis $(f_1,\dots,f_r)$ of $\mc{F}$ corresponds to an $r\times n$ matrix $(a_{ij})_{ij}\in\mathbb{M}_{r\times n}(E)$
whose entries are determined by $f_i=\sum_{j=1}^n a_{ij}X_j^2$ for $1\leq i\leq r$.
Elementary row operations on this $r\times n$-matrix induce a change of the $E$-basis  of $\mc{F}$.
Hence we may find an $E$-basis $(f_1,\dots,f_r)$ such that the corresponding matrix $(a_{ij})_{ij}$ is in reduced row echelon form.
Using a permutation of the columns of the matrix, which corresponds to a permutation of the variables $X_1,\dots,X_n$, we may further reduce to the situation where $a_{ij}=\delta_{ij}$ for $1\leq i,j\leq r$.
 
Let $R=E[X_{r+1},\dots,X_{n}]$, which is an $E$-algebra of Krull dimension $s=n-r$.
We set $\vf_i=X_i^2-f_i$ for $1\leq i\leq r$. Then $\vf_1,\dots,\vf_r\in R$ and
the affine $E$-algebra $$A=R[X_{1},\dots,X_{r}]/(X_{1}^2-\varphi_1,\dots,X_{r}^2-\varphi_r)$$ 
is an integral extension of $R$, hence it also has Krull dimension $s$.
Viewing $\mc{V}$ as an affine variety embedded in $\mathbb{A}_E^n$,
its coordinate ring is equal to $A$, and hence its dimension is $s$.
Therefore the dimension of $\mc{V}$ as a projective variety is $s-1$.

Note that absolute irreducibility of $\mc{V}$ viewed as a projective variety (in $\mathbb{P}_E^{n-1}$) or as an affine variety (in $\mathbb{A}_E^n$) are equivalent.
Hence for the rest of the proof we view $\mc{V}$ as an affine variety.

Recall that quadratic forms in characteristic different from $2$ are absolutely irreducible if and only if their rank is at least $3$.
If $\mc{F}$ contains a quadratic form which is reducible, then the coordinate ring of $\mc{V}$ cannot be a domain and hence $\mc{V}$ must be reducible.
Hence, if $\mc{F}$ contains a quadratic form of rank $1$ or $2$, then $\mc{V}$ is not absolutely irreducible.

Assume now that $\mc{F}$ contains no form of rank $1$ or $2$.
Since $\car(E)\neq 2$ and since $\mc{F}$ consists of diagonal forms, it follows by basic linear algebra (Gauss elimination on the coefficient tuples) that, for any field extension $E'/E$, also the $E'$-vector space of forms in $E'[X_1,\dots,X_n]$ generated by $\mc{F}$ contains no form of rank $1$ or $2$.
Hence we may extend scalars and assume that $E$ is already algebraically closed, and it remains to be shown that $\mc{V}$ in this case is irreducible.

For $1\leq i\leq r$ we have 
$\varphi_i=-\sum_{j=r+1}^n a_{ij}X_j^2$. 
The condition that $\mc{F}$ contains no form of rank $1$ or $2$ implies that 
each of the $s$-tuples $(a_{i\, r+1},\dots,a_{i\, n})$, for $1\leq i\leq r$, 
has at least two nonzero coordinates and that no two of them can be obtained from one another by scaling with elements of $\mg{E}$.
It follows that in $R$ the product $\varphi_1\cdots\varphi_r$ is not divisible by the square of any element of $R\setminus E$.
In particular, no nontrivial product of some of the forms $\varphi_1,\dots,\varphi_r$ is a square in $F=E(X_{r+1},\dots,X_n)$, the fraction field of $R$.
This implies that $F(\sqrt{\varphi_1},\dots,\sqrt{\varphi_r})$ is a field extension of degree $2^r$ of $F$.
It follows that the coordinate ring $A$ of $\mc{V}$ is isomorphic to the subring $R[\sqrt{\varphi_1},\dots,\sqrt{\varphi_r}]$ of $F(\sqrt{\varphi_1},\dots,\sqrt{\varphi_r})$. Hence $A$ is a domain, whereby $\mc{V}$ is irreducible.
\end{proof}

A system $\mc{S}$ of quadratic forms defined on a finite-dimensional $E$-vector space $V$ is \emph{simultaneously diagonalisable} if, for $n=\dim_EV$, there exists an $E$-basis $(v_1,\dots,v_n)$ of $V$ such that $q(v_i,v_j)=q(v_i)+q(v_j)$ holds for all $q\in\mc{S}$ and $1\leq i<j\leq n$ (i.e.,~a simultaneous orthogonal basis for all forms in $\mc{S}$).

\begin{cor}\label{C:diag-systqf-absirred-check}
Let $V$ be an $n$-dimensional $E$-vector space. Let $r\in\nat$ and let $q_1,\dots,q_r:V\to E$ be quadratic forms over $V$ which are simultaneously diagonalisable.
Assume that for any $(a_1,\dots,a_r)\in E^r\setminus\{0\}$, 
the quadratic form $a_1q_1+\dots+a_rq_r$ has rank at least $3$.
Then the equations $q_1=\dots=q_r=0$ on $V$ define an absolutely irreducible projective variety of dimension $n-r-1$.
\end{cor}
\begin{proof}
By taking an $E$-basis $(v_1,\dots,v_n)$ of $V$ such that $q_k(v_i+v_j)=q_k(v_i)+q_k(v_j)$ holds for $1\leq k\leq r$ and $1\leq i<j\leq n$, we identify the system of equations $q_1=\dots=q_r=0$ on $V$ with a system of $r$ diagonal forms in $E[X_1,\dots,X_n]$, and this identification is compatible with scalar extension.
The hypothesis yields that the $E$-vector space of quadratic forms which is generated by $q_1,\dots,q_r$ has dimension $r$ and contains no form of rank $1$ or $2$. 
Hence \Cref{P:pencil-rank-geq3-absolirr} yields the desired statement.
\end{proof}

Recall that $E_f=E[X]/(f)$ for $f\in E[X]$.

\begin{lem}\label{L:transfer-rank}
Let $f\in E[X]$. 
Set $\vartheta=X+(f)\in E_f$.
Let $s:E_f\to E$ be an $E$-linear form and let $r\in\nat$ be such that $s(\vartheta^i)=0$ for $0\leq i<r$ and $s(\vartheta^r)\neq 0$.
Then the rank of the form $q:E_f\to E,x\mapsto s(x^2)$ over $E$ is at least $r+1$.
\end{lem}

\begin{proof}
It follows from the hypothesis that $\mc{B}=(1,\vartheta,\vartheta^2,\dots,\vartheta^r)$ is $E$-linearly independent.
Let $V=\bigoplus_{i=0}^rE\vartheta^i$.
The Gram matrix of the restricted quadratic form $q|_V$ with respect to the $E$-basis $\mc{B}$ of $V$ has 
entries $1$ on the counter-diagonal and $0$ everywhere above the counter-diagonal.
Therefore $q|_V$ is a regular form of rank equal to $\dim_EV=r+1$.
Hence the rank of $q$ is at least $r+1$.
\end{proof}

\begin{lem}\label{L:powerbasis-transfers}
Let $f\in E[X]\setminus E$ be square-free. 
Set  $\vartheta=X+(f)\in E_f$ and $n=\deg(f)$, so that $E_f=E[\vartheta]=\bigoplus_{i=0}^{n-1}E\vartheta^i$.
For $0\leq i\leq n-1$, let $s_i:E_f\to E$ denote the $E$-linear form determined by 
$s_i(\vartheta^j)=\delta_{ij}\mbox{ for }0\leq j\leq n-1\,\,.$
Let $k\in\{1,\dots,n-1\}$, $c_k,\dots,c_{n-1}\in E$ with $c_k\neq 0$ and $s=\sum_{i=k}^{n-1}c_{i}s_{i}$.
Then $$q:E_f\to E, x\mapsto s(x^2)$$ is a quadratic form over $E$ of rank at least ${k+1}$.
Furthermore $$q^\ast:E_f\times E\to E,(x,z)\mapsto q(x)-c_kz^2$$ is a quadratic form over $E$ of rank at least ${k+2}$.
\end{lem}
\begin{proof}
The hypotheses imply that $s(\vartheta^i)=0$ for $0\leq i<k$ and $s(\vartheta^k)=c_k\neq 0$. Hence the claim on $q$ follows by \Cref{L:transfer-rank}, and since $c_k\in\mg{E}$, this implies the statement on $q^\ast$.
\end{proof}

\begin{thm}\label{T:transfer-curve}
Let $f\in E[X]\setminus E$ be separable and $g\in E[X]$ coprime to $f$.
There exists an absolutely irreducible affine curve $\mc{C}$ over $E$ such that the following holds for every field extension $E'/E$:
\begin{quotation}
$\mc{C}(E')\neq \emptyset$ if and only if there exists $a\in E'$ such that
$(X-a)\cdot g$ is a square modulo $f$ in $E'[X]$ and $f(a)\neq 0$.
\end{quotation}
\end{thm}
\begin{proof}
We fix an algebraic closure $E_\alg$ of $E$, and we set $L=E_\alg[X]/(f)$ and $\vartheta=X+(f)\in L$, whereby $L=E_\alg[\vartheta]$.
For any field extension $E'/E$, we set 
\begin{eqnarray*}
\mc{C}'(E') & = & \left\{(c_0,\dots,c_{n-1})\in E'^{n}\,\,\quad\left\vert\,\, g(\vartheta)\cdot\left(\mbox{$\sum$}_{i=0}^{n-1}c_i\vartheta^i\right)^2-\vartheta\in E'\right.\right\}\,\mbox{ and }\\
\mc{C}(E') & = & \left\{(c_0,\dots,c_{n-1})\in \mc{C}'(E')\,\,\left\vert\,\, \N_{E'_f/E'}\left(\mbox{$\sum$}_{i=0}^{n-1}c_i\vartheta^i\right)\neq 0\right.\right\}.
\end{eqnarray*}
Clearly $\mc{C}'$ is an affine $E$-variety, given as a closed subvariety of $\mathbb{A}^n_E$.
We claim that $\mc{C}'$ is an absolutely irreducible curve.
Since $f$ is separable, the $E_\alg$-algebra $L$ is isomorphic to $E_\alg^{\, n}$, and we obtain that $\mg{L}=\sq{L}$.
Since $g$ is coprime to $f$, it follows that $g(\vartheta)\in\mg{L}=\sq{L}$.

We will show that $C'=\mc{C}'(E_\alg)$ is an irreducible curve.
We denote by $C^\ast$ the $E$-Zariski closure of $C'$ in $\mathbb{P}^n_E(E_\alg)$ (i.e.~the variety defined by the homogenisations of polynomials defining $C'$ in $\mathbb{A}^n_E$).
Let $s_1,\dots,s_{n-1}:L\to E_\alg$ be the $E_\alg$-linear forms such that $s_i(\vartheta^j)=\delta_{ij}$ for $1\leq i<n$ and $0\leq j<n$.
Then
 ${C}^\ast$ is given in $\mathbb{P}^n_E({E_\alg})$ as the vanishing set of the following system of $n-1$ quadratic forms defined on the $(n+1)$-dimensional $E_\alg$-vector space $L\times E_\alg$:
\begin{eqnarray*}
q_1:& L\times E_\alg\to E_\alg,&(x,\lambda)\mapsto s_1(g(\vartheta)x^2)-\lambda^2\\
q_i:& L\times E_\alg\to E_\alg,&(x,\lambda)\mapsto s_i(g(\vartheta)x^2)\hspace{1.2cm}(2\leq i\leq n-1)
\end{eqnarray*}

By \Cref{L:powerbasis-transfers}, every nontrivial $E_\alg$-linear combination of the quadratic forms $q_1,\dots,q_{n-1}$ has rank at least $3$.
Since $L\times E_\alg\simeq E_\alg^{n+1}$, the primitive (i.e.~indecomposable) idempotents in $L\times E_\alg$ form an $E_\alg$-basis of $L\times E_\alg$. This $E_\alg$-basis is a common diagonal basis for $q_1,\dots,q_{n-1}$.
Hence the system of quadratic forms $(q_1,\dots,q_{n-1})$ is simultaneously diagonalisable.
Therefore \Cref{C:diag-systqf-absirred-check} shows that ${C}^\ast$ is irreducible, and hence so is $C'$. 

Note that $C=\mc{C}(E_\alg)$ is an affine $E$-Zariski open subset of $C'$.
Furthermore, as $\mg{L}=\sq{L}$, for every $a\in E_\alg$ with $f(a)\neq 0$ there exist
$c_0,\dots,c_{n-1}\in E_\alg$ such that $g(\vartheta)(\sum_{i=0}^{n-1}c_i\vartheta^i)^2=\vartheta-a$.
In particular $C$ is nonempty.
Hence $\mc{C}$ is an absolutely irreducible affine curve over $E$.

Consider finally an arbitrary field extension $E'/E$.
Since $g$ is coprime to $f$, $g(\vartheta)$ is invertible in $E'[\vartheta]$.
For $a\in \mg{E'}$ it follows that $(X-a)\cdot g$ is a square modulo $f$ in $E'[X]$ if and only if 
$\vartheta-a=g(\vartheta)h(\vartheta)^2$ for some $h\in E'[X]$, and in that case we have
 $f(a)\neq 0$ if and only if $h(\vartheta)\in\mg{E'[\vartheta]}$, if and only if $\N_{E'_f/E'}(h(\vartheta))\neq 0$.
From this we conclude that $\mc{C}$ has the properties claimed in the statement.
\end{proof}

\begin{prop}\label{P:linear-square-class-realisation}
Let $f\in E[X]\setminus E$ be separable. Set $\vartheta=X+(f)\in E_f$.
Assume that the affine curve $Y^2=f(X)$ over $E$ has a point of odd degree and, for every $\alpha\in\mg{E_f\!\!}$, there exists 
$a\in E$ such that $\vartheta-a\in\alpha\sq{E_f\!\!}$.
Then $f$ is square-reflexive.
\end{prop}
\begin{proof}
By the hypothesis, there exists $s\in E[X]$ such that $f-s^2$ has a factor $q\in\mc{P}$ of odd degree. 
Hence it suffices to check Condition $(ii')$ of \Cref{SRtoRC}.

Consider an arbitrary $\rho\in\mf R_2(E)$ with $\supp(\rho)\subseteq\supp(f)$. 
Set $\rho'=\rho-\partial(\{f,q\})$ and $c=\lc(f)$.
Then $\supp(\rho')\subseteq \supp(f)\cup\{\infty\}$ and $\rho'_\infty=\{(-1)^{\deg(f)}c\}$.
Let $\alpha\in \mg{E_f\!\!}$ be such that $\rho'_p=\{\alpha_p\}$ for all $p\in\supp(f)$.
By the hypothesis, there exists $a\in E$ such that $\vartheta-a\in\alpha\sq{E_f\!\!}$.
For any $p\in\supp(f)$, we obtain that $\partial_p(\{f,X-a\})=\{\alpha_p\}=\rho'_p$ and thus 
$\partial_p(\{f,q\cdot (X-a)\})=\rho_p$. 
If $q\notin\supp(f)$, then $\partial_q(\{f,q\cdot (X-a)\})=0=\rho_q$, because $f$ is a square modulo $q$.
We set $\rho''=\rho-\partial(\{f,q\cdot (X-a)\})$.
As $(X-a)\cdot q$ is monic of even degree, we have $\rho''_\infty=\rho_\infty=0$.
Hence $\rho''_p=0$ for all $p\in\mc{P}'$ except possibly for $p=X-a$.
As $\rho''\in\mf R_2(E)$, \Cref{L:norm-at-infinity} yields that $\rho''=0$, whereby $\rho=\partial(\{f,q\cdot (X-a)\})$.
\end{proof}

Recall that $E$ is called \emph{pseudo-algebraically closed} if every absolutely irreducible curve over $E$ has an $E$-rational point.

We obtain a new proof for a result due to Efrat \cite{Efr01} in the special case of a rational function field over a pseudo-algebraically closed field. (The result in \cite{Efr01} covers function fields of curves in general over such a base field.)

\begin{thm}\label{C:PAC}
Assume that $E$ is pseudo-algebraically closed.
Then $E$ is square-reflexive. In particular,
$u(E(X))\leq 4$ and quadratic forms over $E(X)$ satisfy the local-global principle for isotropy with respect to $\Omega_{E(X)/E}$.
\end{thm}
\begin{proof}
In view of \Cref{P:123-lgp} and \Cref{LGPI4}, we only need to show that $E$ is square-reflexive.
Assume first that $E$ is perfect.
Let $f\in E[X]$ be square-free. As $E$ is perfect, $f$ is separable.
By the hypothesis, every absolutely irreducible curve over $E$ has an $E$-rational point.
It follows by 
\Cref{T:transfer-curve} that every square-class of $E_f$ is given by 
$\vartheta-a$ for $\vartheta=X+(f)\in E_f$ and some $a\in E$, and by \Cref{P:linear-square-class-realisation} this implies that $f$ is square-reflexive.
This shows $E$ is square-reflexive.

In the general case, let $E'$ denote the perfect closure of $E$.
Note that $E'$ is also pseudo-algebraically closed, and hence $E'$ is square-reflexive, by the above argument.
Since $E$ is of characteristic different from $2$,
$E'/E$ is a direct limit of finite extensions of odd degree.
Since $E'$ is square-reflexive, it follows by \Cref{C:srf-Springer} that $E$ is square-reflexive.
\end{proof}

\begin{rem}
Using \Cref{C:PAC}, we readily get a second proof of \Cref{FFSR}.
Suppose that $E$ is finite of odd cardinality. Let $E'/E$ be an odd closure of $E$. By \cite[Corollary 11.2.4]{FJ}, then $E'$  is pseudo-algebraically closed. (This follows directly from the Hasse-Weil bound.) 
Hence $E'$ is square-reflexive, by \Cref{C:PAC}.
Then \Cref{C:srf-Springer} yields that $E$ is square-reflexive.
\end{rem}

\bibliographystyle{amsalpha}

\end{document}